\newif \ifpreprint
\preprinttrue

\ifpreprint
\documentclass[a4paper,11pt]{amsart}
\else
\documentclass{m2an}
\fi

\ifpreprint
\usepackage[a4paper,left=2.5cm,right=2.5cm,top=2.5cm,bottom=2.5cm]{geometry}
\usepackage[foot]{amsaddr}
\fi

\usepackage{pgfplots}
\newcommand{\SlopeTriangle}[6]
{

    \pgfplotsextra
    {
        \pgfkeysgetvalue{/pgfplots/xmin}{\xmin}
        \pgfkeysgetvalue{/pgfplots/xmax}{\xmax}
        \pgfkeysgetvalue{/pgfplots/ymin}{\ymin}
        \pgfkeysgetvalue{/pgfplots/ymax}{\ymax}

        \pgfmathsetmacro{\xArel}{#1}
        \pgfmathsetmacro{\yArel}{#3}
        \pgfmathsetmacro{\xBrel}{#1-#2}
        \pgfmathsetmacro{\yBrel}{\yArel}
        \pgfmathsetmacro{\xCrel}{\xArel}

        \pgfmathsetmacro{\lnxB}{\xmin*(1-(#1-#2))+\xmax*(#1-#2)} 
        \pgfmathsetmacro{\lnxA}{\xmin*(1-#1)+\xmax*#1} 
        \pgfmathsetmacro{\lnyA}{\ymin*(1-#3)+\ymax*#3} 
        \pgfmathsetmacro{\lnyC}{\lnyA+#4*(\lnxA-\lnxB)}
        \pgfmathsetmacro{\yCrel}{\lnyC-\ymin)/(\ymax-\ymin)} 

        \coordinate (A) at (rel axis cs:\xArel,\yArel);
        \coordinate (B) at (rel axis cs:\xBrel,\yBrel);
        \coordinate (C) at (rel axis cs:\xCrel,\yCrel);

        \draw[#6]   (A)-- node[anchor=north] {#5}
                    (B)--
                    (C)--
                    cycle;
    }
}

\usepackage{amssymb}

\usepackage[breaklinks,bookmarks=false]{hyperref}
\hypersetup{colorlinks,linkcolor=blue,citecolor=blue,urlcolor=blue,plainpages=false,pdfwindowui=false}
\usepackage{mathtools}

\usepackage{bm}

\newtheorem{theorem}{Theorem}
\newtheorem{lemma}[theorem]{Lemma}
\newtheorem{corollary}[theorem]{Corollary}
\newtheorem{remark}[theorem]{Remark}

\numberwithin{equation}{section}
\numberwithin{theorem}{section}

\newcommand{\eq}{:=}
\newcommand\RT{\bm{\mathcal{R\hspace{-0.1em}T}}\hspace{-0.25em}}

\newcommand{\TR}{\textup R}
\newcommand{\TD}{\textup D}
\newcommand{\TK}{\textup K}
\newcommand{\TM}{\textup M}

\newcommand{\JAC}{\mathbb{J}}

\newcommand{\oa}{{\omega_\ba}}
\newcommand{\oK}{{\omega_K}}

\newcommand{\grad}{\boldsymbol \nabla}

\renewcommand{\div}{\grad {\cdot}}

\newcommand{\bgrad}{\grad_h}

\newcommand{\vv}{\boldsymbol v}

\newcommand{\ddiv}{\operatorname{div}}

\newcommand{\CJ}{\mathcal J}
\newcommand{\CI}{\mathcal I}

\newcommand{\bII}{\CI_h}
\newcommand{\bpi}{\pi_h}

\newcommand{\bn}{\boldsymbol n}
\newcommand{\jmp}[1]{[\![#1]\!]}

\newcommand{\haK}{\tau^{\ba}_K}

\newcommand{\ha}{h_{\ba}}

\newcommand{\BH}{\boldsymbol H}
\newcommand{\BL}{\boldsymbol L}
\newcommand{\BW}{\boldsymbol W}
\newcommand{\BX}{\boldsymbol X}

\newcommand{\ba}{{\boldsymbol a}}
\newcommand{\bb}{\boldsymbol b}
\newcommand{\br}{\boldsymbol r}
\newcommand{\bv}{\boldsymbol v}
\newcommand{\bw}{\boldsymbol w}
\newcommand{\bx}{\boldsymbol x}
\newcommand{\by}{\boldsymbol y}
\newcommand{\bzero}{\boldsymbol 0}

\newcommand{\CP}{\mathcal P}
\newcommand{\CT}{\mathcal T}
\newcommand{\CF}{\mathcal F}
\newcommand{\CV}{\mathcal V}

\newcommand{\BCP}{\boldsymbol \CP}

\newcommand{\po}{\widehat \psi_{\bzero}}
\newcommand{\hu}{\widehat u}
\newcommand{\hv}{\widehat v}
\newcommand{\hs}{\widehat s}
\newcommand{\oo}{\widehat \omega_{\bzero}}
\newcommand{\CTo}{\widehat \CT_{\bzero}}

\newcommand{\pa}{\psi_{\ba}}
\newcommand{\la}{\lambda_{\ba}}
\newcommand{\lo}{\widehat \lambda_{\bzero}}
\newcommand{\ra}{\rho_{\ba}}
\newcommand{\CTa}{\CT_{\ba}}

\newcommand{\CFoi}{\widehat \CF_{\bzero}^{\rm i}}
\newcommand{\CFo}{\widehat \CF_{\bzero}}

\newcommand{\hK}{\widehat K}

\newcommand{\phig}{\phi^{\rm g}}
\newcommand{\phid}{\phi^{\rm d}}

\newcommand{\REF}{\widehat R}

\begin{document}

\ifpreprint
\title{A quasi-interpolation operator yielding fully computable error bounds}
\author{T. Chaumont-Frelet$^\star$ and M. Vohral\'ik$^{\dagger,\ddagger}$}

\address{\vspace{-.5cm}}
\address{\noindent \tiny \textup{$^\star$Inria, Univ. Lille, CNRS, UMR 8524 -- Laboratoire Paul Painlev\'e, 40 Av. Halley, 59650 Villeneuve-d'Ascq, France}}
\address{\noindent \tiny \textup{$^\dagger$Inria, 48 rue Barrault, 75647 Paris, France}}
\address{\noindent \tiny \textup{$^\ddagger$CERMICS, Ecole nationale des ponts et chauss\'ees, IP Paris, 77455 Marne-la-Vall\'ee, France}}

\begin{abstract} 
We design a quasi-interpolation operator from the Sobolev space $H^1_0(\Omega)$
to its finite-dimensional finite element subspace formed by piecewise polynomials
on a simplicial mesh with a computable approximation constant. The operator 1)
is defined on the entire $H^1_0(\Omega)$, no additional regularity is needed;
2) allows for an arbitrary polynomial degree; 3) works in any space dimension;
4) is defined locally, in vertex patches of mesh elements; 5) yields optimal estimates
for both the $H^1$ seminorm and the $L^2$ norm error; 6) gives a computable constant
for both the $H^1$ seminorm and the $L^2$ norm error; 7) leads to the equivalence of
global-best and local-best errors; 8) possesses the projection property.
Its construction follows the so-called potential reconstruction from a posteriori
error analysis. Numerical experiments illustrate that our quasi-interpolation operator
systematically gives the correct convergence rates in both the $H^1$ seminorm and
the $L^2$ norm and its certified overestimation factor is rather sharp and stable
in all tested situations.

\vspace{.25cm}
\end{abstract}
\fi

\maketitle

\section{Introduction}

Due to their ability to operate on general geometries,
finite element methods based on unstructured meshes have
become very popular to discretize boundary value problems
over the past decades~\cite{ciarlet_2002a,Ern_Guermond_FEs_I_21}.
The approximation properties of the finite element spaces are
one of the key factors that eventually govern the accuracy
of the resulting numerical schemes. Such approximation properties
are often obtained thanks to (quasi-)interpolation operators, i.e.,
mappings that associate to a given function $u$ from the
infinite-dimensional Sobolev space $H^1_0(\Omega)$ an explicitly constructed
approximation $\CI_h u$ in the finite-dimensional finite element space.

Given a polytopal domain $\Omega \subset \mathbb R^d$, $d \geq 1$,
triangulated by a simplicial mesh $\CT_h$, this work focuses on conforming
finite elements of polynomial degree $p \geq 1$, $\CP_p(\CT_h) \cap H^1_0(\Omega)$,
forming a finite-dimensional subspace of the Sobolev space $H^1_0(\Omega)$. In this context,
a natural way to associate to a function $u \in H^1_0(\Omega)$ an interpolant
$\CI_h u \in \CP_p(\CT_h) \cap H^1_0(\Omega)$ is by requiring
that $(\CI_h u)(\bx_\ell) = u(\bx_\ell)$ for suitably
chosen interpolation nodes $\{\bx_\ell\}_\ell$. This process is known as Lagrange interpolation,
and presents the advantage that it is fully local: For a mesh element $K \in \CT_h$, the
interpolation error $(u-\CI_h u)|_K$ will only depend on $u|_K$. Besides, fully-explicit error
bounds expressed in terms of Sobolev semi-norms of $u$ are available in the
literature~\cite{arcangeli_gout_1976a,Cars_Ged_Rim_expl_cnst_12,Hao_Guan_Mao_Chen_cnsts_FEs_21,Kik_Liu_cnsts_P0_P1_07,kobayashi_tsuchiya_2014a,Kob_interp_Lag_15,liu_kikuchi_2010a,Liu_You_cnsts_P2_18}.

The main drawback of Lagrange interpolation, however, is that
it requires the nodal values of $u$ to be well-defined, which
is not the case (expect in one space dimension, $d=1$) if $u$ barely
sits in $H^1_0(\Omega)$. When considering, e.g., the Poisson problem
in a convex domain with an $L^2(\Omega)$ right-hand side, this is
not an issue, since then the solution $u$ belongs to $H^2(\Omega)$,
enabling the use of Lagrange interpolation. However, the approach fails
when considering more complex geometries and/or coefficients.
Additionally, some applications, such as a posteriori error
estimation~\cite{carstensen_1999a,melenk_2005a}, preconditioning~\cite{schoberl_melenk_pechstein_zaglmayr_2008a}, or localized orthogonal
decomposition methods~\cite{maalqvist_peterseim_2014a},
crucially hinge on interpolation operators defined for generic
$H^1_0(\Omega)$ target functions (i.e. with
the minimal variational regularity of the PDE under consideration).

There is therefore a need to define interpolation operators for non-smooth functions.
The resulting operators are often called ``quasi-interpolation'' operators, since
they do not, strictly-speaking, interpolate the target at any point. The seminal contribution
in this direction is the work of Cl\'ement~\cite{clement_1975a}, where an operator
defined over the whole $L^1(\Omega)$ space is introduced with optimal approximation properties.
The Cl\'ement operator does not preserve boundary conditions, and it is not a projection,
i.e., $\CI_h u_h \neq u_h$ for $u_h$ already in the finite element space.
These two properties have been incorporated later on in the key contribution of
Scott and Zhang~\cite{scott_zhang_1990a}. For these two operators, the degrees of freedom (dofs)
of the interpolant $\CI_h u$ are fixed by suitably weighted mean values of the target function
$u$. Another family of quasi-interpolation operators is due to Oswald~\cite{oswald_1999a}.
There, a discontinuous approximation is first built by
projecting element-wise $u|_K$ onto a local polynomial space $\CP_p(K)$
of the mesh cell $K \in \CT_h$.
The dofs of this approximation that are shared by several elements are then
suitably averaged to produce a conforming interpolant
$\CI_h u \in \CP_p(\CT_h) \cap H^1_0(\Omega)$.
Since then, these constructions have been improved in several ways,
and the design of quasi-interpolation operators is still an active field of
research~\cite{ern_guermond_2017a,Gaw_Holst_Licht_loc_approx_FEEC_21,falk_winther_2014a,hiptmair_pechstein_2019,licht_2023a,schoberl_2001a,Voh_loc_glob_H1_25}.

Because the nodal evaluation of a generic target function $u \in H^1_0(\Omega)$
is not possible, quasi-interpolation operators are not fully local.
Indeed, the value of dofs needs to be obtained through some suitable averaging,
meaning that the definition $(\CI_h u)|_K$ will not only depend on $u|_K$, but rather on
$u|_{\oK}$, where $\oK$ is the domain formed by other mesh elements surrounding $K$. Although
this is not a huge drawback in practice, this fact makes the error analysis more complicated
than for standard interpolation operators, since it is not possible to use arguments involving
a single ``reference element''. As a result, a shape-regularity requirement involving $\oK$
rather than each element $K$ individually often appears. In fact, the only fully-explicit
approximation results we are aware of in this direction
are~\cite{carstensen_funken_2000a,verfurth_1999a,Vees_Verf_expl_res_09,Vees_Verf_Poin_stars_12}
but the scope of these results is limited,
as they are specifically established with applications to a posteriori error estimation in mind.

Another important topic related to our work and connected to quasi-interpolation
operators is the comparison between local-best and global-best approximations.
Namely, we ask whether there exists a constant $C > 0$ such that
\begin{equation}
\label{eq_glob_loc_intro}
\min_{u_h \in \CP_p(\CT_h) \cap H^1_0(\Omega)} \|\grad(u-u_h)\|_\Omega
\leq
C \min_{v_h \in \CP_p(\CT_h)} \|\bgrad(u-v_h)\|_\Omega.
\end{equation}
In other words, we ask whether, up to a generic constant, the conforming
finite element approximation is as good as the element-wise
broken polynomial approximation of the target function. It is in fact clear
that such a constant exists, because the left-hand side vanishes whenever
the right-hand side does. However, the dependence of the constant $C$ on
key discretization parameters is a subtle issue.
Following~\cite{Aur_Fei_Kem_Pag_Praet_loc_glob_13,Cars_Pet_Sched_comp_FEs_12,veeser_2016a},
recent results in this direction are given in~\cite{Gaw_Holst_Licht_loc_approx_FEEC_21,licht_2023a,Voh_loc_glob_H1_25}.
Notice that once an inequality such as~\eqref{eq_glob_loc_intro} is established, approximability
estimates for the finite element space $\CP_p(\CT_h) \cap H^1_0(\Omega)$ follow,
since the right-hand side may be easily estimated in terms of (broken) Sobolev
norms of $u$ using element-wise Poincar\'e inequalities (see Section
\ref{section_local_best} below).

In this work, we propose a new quasi-interpolation operator
$\CJ_h^p: H^1_0(\Omega) \to \CP_p(\CT_h) \cap H^1_0(\Omega)$.
This operator can be computed locally by solving patch-wise
finite element projection problems,
works for any polynomial degree $p \geq 1$, preserves boundary conditions,
is a projection, and, moreover, satisfies
\begin{equation}
\label{eq_error_intro}
\|\grad(u-\CJ_h^p(u))\|_\Omega
\leq
C \min_{v_h \in \CP_p(\CT_h)} \|\bgrad(u-v_h)\|_\Omega
\qquad
\forall u \in H^1_0(\Omega)
\end{equation}
for a constant $C$ only depending on the shape-regularity parameter
of the mesh $\CT_h$, the dimension $d$, and the polynomial degree $p$. In particular,
our interpolation operator has the approximation power of discontinuous
piecewise polynomials and leads to the local--global equivalence~\eqref{eq_glob_loc_intro}.
We also establish localized versions of~\eqref{eq_error_intro}, as well as error estimates in
the $L^2(\Omega)$ norm (see Section~\ref{section_main_results} below).
Our construction is inspired by the so-called potential reconstructions used in the context of
a posteriori error estimation for nonconforming and mixed finite element methods~\cite{ern_vohralik_2015a}, as well as recent works of the authors on commuting quasi-interpolation
operators~\cite{chaumontfrelet_vohraik_2022a,ern_gudi_smears_vohralik_2022a}.

The main novelty of our work is that the constant $C$ appearing in~\eqref{eq_error_intro}
is fully computable. Specifically, it is calculated through an algorithm that
amounts to solving small, uncoupled, matrix eigenvalue problems that stem
from patch-wise finite element spaces. We also establish localized versions
of~\eqref{eq_error_intro} and $L^2(\Omega)$ error estimates with fully-computable constants.
To the best of our knowledge, this work is the first to provide computable constants
under minimal regularity, for arbitrary polynomial degree, and yielding
the comparison of global-best and local-best errors.
We emphasize that the proposed algorithm computes the generic constant $C$
in~\eqref{eq_error_intro} valid for all $u \in H^1_0(\Omega)$ (without needing to know $u$).

The remainder of this work is organized as follows. Section~\ref{section_settings}
gives the setting and recalls key preliminary results. 
In Section~\ref{section_quasi_interpolation}, we present elementwise Lagrange interpolation,
the local-best approximation, and motivate our approach. The construction of our
quasi-interpolation operator and statement of its key properties including~\eqref{eq_error_intro}
forms the content of Section~\ref{section_main_results}.
Finally, Section~\ref{section_proof} collects the proofs and Section~\ref{sec_num} reports on
the actual behavior of our quasi-interpolation operator in several numerical experiments.

\section{Meshes, spaces, patches, and Poincar\'e inequality}
\label{section_settings}

\subsection{Computational mesh}\label{sec_mesh}

We consider an open, bounded, connected, Lipschitz polyhedral domain $\Omega\subset \mathbb R^d$,
$d \geq 1$, partitioned into a mesh $\CT_h$ that consists of (open) simplicial elements $K$.
We assume that the mesh is matching in the sense that the intersection $\overline{K_+} \cap \overline{K_-}$
of two elements $K_\pm \in \CT_h$ is either empty or a single common $d'$-dimensional
subsimplex of $\overline{K_+}$ and $\overline{K_-}$, $0 \leq d' \leq d-1$
(a single vertex, edge, or face of $K_-$ and $K_+$ if $d=3$). This assumption is standard
(see, e.g., \cite[Section 2.2]{ciarlet_2002a} or \cite[Definition 6.11]{Ern_Guermond_FEs_I_21})
We denote by $\CV_h$ the set of vertices of $\CT_h$.

For a simplex $K \subset \mathbb R^d$,
$h_K$ denotes the diameter of $K$ and $\rho_K$ the diameter 
of the largest ball contained in $\overline{K}$.
The shape-regularity parameter of $K$ is then defined by 
\begin{equation*}
\kappa_K \eq \frac{h_K}{\rho_K}.
\end{equation*}
If $\CT \subset \CT_h$ is a submesh, the shape-regularity parameter of $\CT$ is
$\kappa_{\CT} \eq \max_{K \in \CT} \kappa_K$. Notice that $\kappa_K \geq 1$
and that it is possible to design strongly graded meshes $\CT_h$ while maintaining
$\kappa_{\CT_h}$ bounded, as long as the elements remain isotropic.

If $K \subset \mathbb R^d$ is a simplex, we employ the notation $\CF_K$ for its
$(d-1)$-dimensional faces and $\CV_K$ for its vertices.
Then, if $\ba \in \CV_K$ is a vertex of $K$, the notation $\haK$ stands for the
distance between $\ba$ and the (hyper)plane generated by the face of $K$ opposite to $\ba$. 
Notice that $\rho_K \leq \haK \leq h_K$, so that in particular $h_K/\haK \leq \kappa_K$.

\subsection{Sobolev spaces}

Throughout this manuscript, if $\omega \subset \Omega$ is an open, bounded, connected,
Lipschitz subset of $\Omega$, then
$L^2(\omega)$ is the Lebesgue space of scalar-valued square-integrable functions and
$\BL^2(\omega) \eq [L^2(\omega)]^d$ is the space of vector-valued square-integrable functions.
We denote by $(\cdot,\cdot)_\omega$ and $\|\cdot\|_\omega$ the inner products and norms
of both spaces. 
We employ the notation $L^\infty(\omega)$ and $\BL^\infty(\omega)$ for
essentially bounded scalar- and vector-valued functions, and denote by
$\|{\cdot}\|_{L^\infty(\omega)}$ and $\|{\cdot}\|_{\BL^\infty(\omega)}$ their usual norms.

The Sobolev spaces
\begin{equation*}
H^1(\omega) \eq \left \{
v \in L^2(\omega) \; | \; \grad v \in \BL^2(\omega)
\right \},
\quad
\BH(\ddiv,\omega) \eq \left \{
\bw \in \BL^2(\omega) \; | \; \div \bw \in L^2(\omega)
\right \},
\end{equation*}
where $\grad$ and $\div$ respectively denote the gradient and divergence operators
defined in the sense of distributions, will be useful.
We will also employ the following notation:
\begin{equation*}
H^1_0(\omega) \eq \left \{
v \in H^1(\omega) \; | \; v = 0 \text{ on } \partial \omega
\right \}, \quad
\BH(\ddiv^0,\omega) \eq \left \{
\vv \in \BH(\ddiv,\omega) \; | \; \div \vv = 0
\right \},
\end{equation*}
where the boundary value is understood in the trace sense.

For scalar-valued functions, we will also use higher-order Sobolev spaces.
Namely for an integer $s \in \mathbb N$, $H^s(\omega)$ stands for 
the space of functions $v \in L^2(\omega)$ such that
\begin{equation*}
\partial_{\boldsymbol \alpha} v \in L^2(\omega)
\end{equation*}
for all $\boldsymbol \alpha \in \mathbb N^d$ with
$|\boldsymbol \alpha|_1 \leq s$, whereby $|{\cdot}|_1$
denotes the $\ell_1$ norm on $\mathbb N^d$. We equip $H^s(\omega)$
with the seminorm
\begin{equation}\label{eq_Hs_sn}
|v|_{H^s(\omega)}^2
\eq
\sum_{\substack{\boldsymbol \alpha \in \mathbb N^d \\ |\boldsymbol \alpha|_1 = s}}
\|\partial_{\boldsymbol \alpha} v\|_{\omega}^{2}
\qquad
\forall v \in H^s(\omega).
\end{equation}

The broken Sobolev spaces
\begin{equation*}
H^s(\CT_h) = \left \{
v \in L^2(\Omega) \; | \; v|_K \in H^s(K) \quad \forall K \in \CT_h
\right \}
\end{equation*}
will also be useful. The broken gradient is in particular defined element-wise as
\begin{equation*}
\left (\bgrad v\right )|_K \eq \grad(v|_K) \quad \forall K \in \CT_h
\end{equation*}
for all $v \in H^1(\CT_h)$. It maps $H^1(\CT_h)$ into $\BL^2(\Omega)$.

\subsection{Finite element spaces}

If $K \subset \mathbb R^d$ is a simplex and $q \geq 0$, we denote by $\CP_{q}(K)$ the set
of polynomial functions of total degree at most $q$ on $K$ and
$\BCP_{q}(K) \eq [\CP_{q}(K)]^d$.
The set of Raviart--Thomas polynomials on $K$ is then
$\RT_{q}(K) \eq \bx \CP_{q}(K) + \BCP_{q}(K)$,
see \cite{nedelec_1980a,raviart_thomas_1977a}.

If $\CT$ is a set of simplices with the corresponding domain $\omega \subset \Omega$,
then $\CP_{q}(\CT)$ collects the functions $v: \omega \to \mathbb R$ such that
$v|_K \in \CP_{q}(K)$ for all $K \in \CT$. The elementwise (broken) space $\RT_{q}(\CT)$
is defined analogously.

Below, we fix a polynomial degree $p \geq 1$.

\subsection{Hat functions}

Throughout this work, the hat functions $\{\pa\}_{\ba \in \CV_h}$
will play an important role. For each mesh vertex $\ba \in \CV_h$, $\pa$
is the only element of $\CP_1(\CT_h) \cap H^1(\Omega)$ such that
$\pa(\bb) = \delta_{\ba,\bb}$ for all $\bb \in \CV_h$, where $\delta$ is the Kronecker symbol.
If $K \in \CT_h$ and $\ba \in \CV_K$, then
\begin{equation}
\label{eq_pa}
\|\pa\|_{L^\infty(K)} = 1
\qquad
\|\grad \pa\|_{\BL^\infty(K)}
=\frac{1}{\haK}.
\end{equation}
We also denote by $\oa$ the open domain corresponding to the support of $\pa$.
Crucially, the hat functions form a partition of unity as
\begin{equation}
\label{eq_PU}
\sum_{\ba \in \CV_h} \pa = 1.
\end{equation}

\subsection{Vertex patches}

If $\ba \in \CV_h$ we define the patch
\begin{equation*}
\CTa = \left \{ K \in \CT_h \; | \; \ba \in \CV_K \right \},
\end{equation*}
of elements having $\ba$ as a vertex. The elements $K \in \CTa$ correspond
to the support $\oa$ of the hat function $\pa$.
For a vertex patch $\CTa$, we define
\begin{equation}
\label{eq_ha}
h_{\ba} \eq \max_{K \in \CT_h | \ba \in \CV_K} h_K.
\end{equation}

\subsection{Poincar\'e inequality}

Since simplices are convex, for all $K \in \CT_h$ and for all $v \in H^1(K)$
such that $(v,1)_K = 0$, we have
\begin{equation}
\label{eq_poincare}
\|v\|_K \leq \frac{h_K}{\pi} \|\grad v\|_K,
\end{equation}
see, e.g., \cite{bebendorf_2003a,payne_weinberger_1960a}.

\section{Elementwise Lagrange interpolation, local-best approximation, and main ideas}
\label{section_quasi_interpolation}

In this section, we present the elementwise Lagrange interpolation operator and discuss
local-best approximation by $H^1$ seminorm elementwise orthogonal projection. We then
motivate the construction of our quasi-interpolation operator.

\subsection{Elementwise Lagrange interpolation operator}
\label{sec_Ip}

We start by recalling the standard Lagrange interpolation operator on each mesh element
$K \in \CT_h$. We will only apply it to functions that are polynomials of degree $p+1$
on $K$, so that we have $\CI^p_K: \CP_{p+1}(K) \to \CP_p(K)$,
\begin{subequations}\label{eq_Ip}
\begin{equation}\label{eq_Ip_K}
(\CI^p_K v_h) (\bx_\ell) = v_h (\bx_\ell)
\end{equation}
for all Lagrange interpolation nodes $\{\bx_\ell\}_\ell$ on $K$ see, e.g.,
\cite[Section~7.4]{Ern_Guermond_FEs_I_21}. Then, the elementwise Lagrange interpolation operator
$\bII^p: \CP_{p+1}(\CT_h) \to \CP_p(\CT_h)$ is defined by
\begin{equation}\label{eq_Ip_Th}
\left .\left (\bII^p v\right )\right |_K \eq \CI_K^p(v|_K) \quad \forall K \in \CT_h,
\end{equation}\end{subequations}
for $v \in \CP_{p+1}(\CT_h)$. Let us point out that if
$v \in \CP_{p+1}(\CT_h) \cap H^1_0(\Omega)$, then
$\bII^p v \in \CP_p(\CT_h) \cap H^1_0(\Omega)$ and that $\bII^p$ is a projection.

\begin{remark}[Other interpolation operators]
We focus on the Lagrange interpolation operator $\CI_K^p$ on each mesh element $K$
for the sake of simplicity.
In fact, any affine-equivalent interpolation operator $\CI_K^p: \CP_{p+1}(K) \to \CP_p(K)$ with
the property that $\bII^p$ maps $\CP_{p+1}(\CT_h) \cap H^1_0(\Omega)$ into
$\CP_p(\CT_h) \cap H^1_0(\Omega)$ can be used.
\end{remark}

\subsection{Local-best approximation}
\label{section_local_best}

Another key ingredient of our quasi-interpolation operator is the so-called
``local-best'' approximation.
For $K \in \CT_h$ and $v \in H^1(K)$, we denote by $\pi_K^p v \in \CP_p(K)$
the $H^1$-orthogonal projection of $v$. Specifically, it is defined as the
only element of $\CP_p(K)$ such that $(\pi_K^p v,1)_K = (v,1)_K$ and
\begin{subequations}
\label{eq_LB}
\begin{equation}
\label{eq_LB_K}
(\grad (\pi_K^p v),\grad q_h)_K = (\grad v,\grad q_h)_K \qquad \forall q_h \in \CP_p(K).
\end{equation}

The ``local-best'' approximation of $v \in H^1(\CT_h)$, $\bpi^p v \in \CP_p(\CT_h)$,
is then defined by setting 
\begin{equation}
\label{eq_LB_T}
(\bpi^p v)|_K \eq \pi_K^p (v|_K) \qquad \forall K \in \CT_h.
\end{equation}\end{subequations}
Notice that then
\begin{equation*}
\|\bgrad(v-\bpi^p v)\|_{K}
=
\min_{v_h \in \CP_p(K)} \|\bgrad(v-v_h)\|_{K}.
\end{equation*}
and
\begin{equation*}
\|\bgrad(v-\bpi^p v)\|_{\Omega}
=
\min_{v_h \in \CP_p(\CT_h)} \|\bgrad(v-v_h)\|_{\Omega}.
\end{equation*}
We stress that, in general, $\bpi^p v$ is discontinuous, i.e., $\bpi^p v \not \in H^1_0(\Omega)$
even if $v \in H^1_0(\Omega)$.

\subsection{Main idea of the construction of the quasi-interpolation operator}
\label{section_main_ideas}

Here, we explain the spirit of the definition of our operator.
Given $u \in H^1_0(\Omega)$, we start by considering the local-best approximation
$\bpi^p u \in \CP_p(\CT_h)$ of $u$ given by~\eqref{eq_LB}. This function is locally
defined and exhibits the best-possible approximation properties, but it is unfortunately
nonconforming since $\bpi^p u \notin H^1_0(\Omega)$ in general.

The next stage is to then to somehow ``locally project'' $\bpi^p u$ to the conforming Lagrange
finite element space $\CP_p(\CT_h) \cap H^1_0(\Omega)$. The optimal projection
\begin{equation}
\label{eq_global_best_bad}
\min_{v_h \in \CP_p(\CT_h) \cap H^1_0(\Omega)} \|\bgrad(\bpi^p u-v_h)\|_\Omega
=
\min_{v_h \in \CP_p(\CT_h) \cap H^1_0(\Omega)} \|\grad(u-v_h)\|_\Omega
\end{equation}
is not satisfactory for a quasi-interpolation operator since it is global: it
requires a global system solve and does not lead to local approximation properties (the
minimizer from~\eqref{eq_global_best_bad} restricted to $K \in \CT_h$
depends on the values of $u$ in the whole domain $\Omega$ in general). The key idea is
therefore to first localize $\bpi^p u$ and then project it locally in vertex patches.

We employ the partition of unity by the hat functions~\eqref{eq_PU} and try to
approximate $\pa \bpi^p u$ into a local contribution $s_h^{\ba} \in \CP_p(\CTa) \cap H^1_0(\oa)$.
This choice makes sense, since (a) $\pa \bpi^p u$ is expected to be close to $\pa u$
which sits in $H^1_0(\oa)$ and (b) the boundary conditions ensure that the object
resulting from the summation of the local contributions is globally conforming.

At this point, it is therefore tempting to define
\begin{equation}
\label{eq_bad_local_approximation}
s_h^{\ba} \eq \arg \min_{v_h \in \CP_p(\CTa) \cap H^1_0(\oa)}
\|\bgrad(\pa \bpi^p u - v_h)\|_\oa, \quad
\CJ_h^{p}(u) \eq \sum_{\ba \in \CV_h} s_h^\ba.
\end{equation}
The construction in~\eqref{eq_bad_local_approximation} faithfully conveys
the ideas behind the construction of our projector. However, as we shall
see later, \eqref{eq_bad_local_approximation} would fail because we are
trying to represent a polynomial of degree $p+1$ with a polynomial of
degree $p$. Indeed, the presence of the hat function $\pa$ increases
the polynomial degree by $1$, and to bring it back, we will use the elementwise
Lagrange interpolation operator \eqref{eq_Ip}.

\subsection{Main idea for the computable approximation constant}
\label{section_ideas_cnst}

The local problem~\eqref{eq_bad_local_approximation} is an approximation of
the piecewise polynomial but discontinuous (nonconforming) datum $\pa \bpi^p u$
by the conforming finite element method. Our key idea to obtain a computable
error bound is to employ duality, i.e., define a local counterpart of
problem~\eqref{eq_bad_local_approximation} via Raviart--Thomas finite elements. 
This will allow to measure the nonconformity of the local-best approximation $\bpi^p u$
vertex patch by vertex patch via a local discrete eigenvalue problem and bring us
to the precision of the $H^1$-orthogonal projection of~\eqref{eq_LB}. This projection
is defined separately on each mesh element and its approximation bound comes with a
fully explicit, computable, constant, see~\eqref{eq_repeated_poincare} below.
Combining these ingredients then gives estimates optimal in the mesh size $h$
for both the $H^1$ seminorm and the $L^2$ norm with a fully computable constant.

\section{Main results}\label{section_main_results}

In this section, we construct our quasi-interpolation operator and present its fully
computable error bounds. The proofs are postponed until Section~\ref{section_proof}. 

\subsection{Local potential reconstruction}

Let a vertex $\ba \in \CV_h$ be given. We develop further~\eqref{eq_bad_local_approximation}.
As it will become apparent below, we need here to start from an arbitrary piecewise polynomial
$u_h \in \CP_p(\CTa)$. We define $s_h^{\ba}(u_h)$ as the continuous piecewise polynomial with
vanishing trace on the boundary $\partial \oa$ of the vertex patch subdomain $\oa$,
$s_h^{\ba}(u_h) \in \CP_p(\CTa) \cap H^1_0(\oa)$, by
\begin{subequations}\label{eq_sha}\begin{equation}
\label{eq_sha_min}
s_h^{\ba}(u_h) \eq \arg \min_{v_h \in \CP_p(\CTa) \cap H^1_0(\oa)}
\|\bgrad(\bII^p(\pa u_h)-v_h)\|_\oa.
\end{equation}
Equivalently, $s_h^{\ba}$ is given by the Euler--Lagrange conditions of~\eqref{eq_sha_min}, reading as: find $s_h^{\ba}(u_h) \in \CP_p(\CTa) \cap H^1_0(\oa)$ such that
\begin{equation}
\label{eq_sha_EL}
(\grad s_h^\ba(u_h),\grad v_h)_\oa = (\bgrad (\bII^p(\pa u_h)),\grad v_h)_\oa
\qquad
\forall v_h \in \CP_p(\CTa) \cap H^1_0(\oa).
\end{equation}\end{subequations}
Since $\|\grad{\cdot}\|_\oa$ is a norm on $H^1_0(\oa)$, $s_h^{\ba}(u_h)$ is indeed uniquely defined.
We observe that the Lagrange interpolation operator $\bII^p$ of~\eqref{eq_Ip} is employed
element-wise in order to reduce the polynomial degree $p+1$ of $\pa u_h$ back to $p$. This 
is legal, since $\pa u_h$ is a polynomial (it would have failed if immediately applied
to $\pa u$). We implicitly extend $s_h^\ba(u_h)$ by zero outside of $\oa$, so that 
$s_h^\ba$ is an operator from $\CP_p(\CTa)$ to $\CP_p(\CT_h) \cap H^1_0(\Omega)$.

\subsection{Local flux reconstruction}

For each vertex $\ba \in \CV_h$, consider the local finite element space
\begin{equation} \label{eq_Wha}
\BW^{\ba}_h \eq \RT_p(\CTa) \cap \BH(\ddiv^0,\oa)
\end{equation}
of divergence-free $\BH(\ddiv,\oa)$-conforming Raviart--Thomas piecewise polynomials
defined over the vertex patch subdomain $\oa$.
For an arbitrary piecewise polynomial $u_h \in \CP_p(\CTa)$, we define
$\br_h^\ba(u_h)$ as the only element of $\BW_h^{\ba}$ such that
\begin{subequations}\label{eq_bra}
\begin{equation}
\label{eq_bra_min}
\br_h^\ba(u_h) \eq \arg \min_{\bw_h \in \BW_h^\ba}
\|\bgrad(\pa u_h)-\bw_h\|_\oa,
\end{equation}
or, equivalently, $\br_h^\ba(u_h) \in \BW_h^{\ba}$ such that
\begin{equation}
\label{eq_bra_EL}
(\br_h^\ba(u_h),\bw_h)_\oa = (\bgrad(\pa u_h),\bw_h)_\oa \qquad \forall \bw_h \in \BW_h^\ba.
\end{equation}\end{subequations}
Notice that we then have
\begin{equation}
\label{eq_bra_sup}
\|\br_h^\ba(u_h)\|_\oa
=
\max_{\substack{\bw_h \in \BW_h^\ba \\ \|\bw_h\|_\oa = 1}}
(\bgrad(\pa u_h),\bw_h)_\oa.
\end{equation}
From a practical standpoint, given $u_h \in \CP_p(\CTa)$, $\br_h^{\ba}(u_h)$
can be computed as $\br_h^{\ba}(u_h) = \br_h$, where
$(\br_h,q_h) \in \BX_h^{\ba} \times Y_h^{\ba}$ is the unique pair such that
\begin{equation}
\label{eq_bra_lagrange_multiplier}
\left \{
\begin{array}{rcl}
(\br_h,\bv_h)_\oa + (q_h,\div \bv_h)_{\oa} &=& (\bgrad(\pa u_h),\bv_h)_\oa,
\\
                    (\div \br_h,p_h)_{\oa} &=& 0,
\end{array}
\right .
\end{equation}
for all $(\bv_h,p_h) \in \BX_h^{\ba} \times Y_h^{\ba}$, with
\begin{equation*}
\BX_h^{\ba} \eq \RT_p(\CTa) \cap \BH(\ddiv,\oa),
\qquad
Y_h^{\ba} \eq \CP_p(\CTa).
\end{equation*}
Indeed, although~\eqref{eq_bra} and~\eqref{eq_bra_lagrange_multiplier} are equivalent,
in contrast to $\BW_h^{\ba}$, both $\BX_h^{\ba}$ and $Y_h^{\ba}$ can be easily equipped
with an explicit basis. 

\subsection{A local primal--dual quotient and its computation by a local matrix eigenvalue problem}

Consider a vertex $\ba \in \CV_h$. For $u_h \in \CP_p(\CTa)$,
let $s_h^{\ba}(u_h)$ be defined by~\eqref{eq_sha} and
$\br_h^{\ba}(u_h)$ by~\eqref{eq_bra}. We define the primal--dual quotient
\begin{equation}\label{eq_lmbd}
\la \eq \max_{u_h \in \CP_p(\CTa)}
\frac{\|\bgrad(\bII^p(\pa u_h)-s_h^\ba(u_h))\|_\oa}{\|\br^\ba(u_h)\|_\oa}.
\end{equation}
Theorem~\ref{thm_la} below warrants that $\la$ is finite and only depends
on the polynomial degree $p$, the space dimension $d$, and the shape-regularity
parameter $\kappa_{\CTa}$ of the patch.

The number $\la$ can be computed through a small-size matrix eigenvalue
problem. Specifically, given a basis of $\CP_p(\CTa)$, since the operator
\begin{equation*}
u_h \in \CP_p(\CTa) \to \CP_p(\CTa) \ni \bII^p(\pa u_h) - s_h^{\ba}(u_h)
\end{equation*}
is linear, its action may be encoded as a matrix
$\TD \in \mathbb R^{|\CP_p(\CTa)| \times |\CP_p(\CTa)|}$.
Similarly, the action of the operator
$\br_h^{\ba}: \CP_p(\CTa) \to \BW_h^{\ba} \subset \BX_h^{\ba}$
can be easily encoded as a simple matrix $\TR \in \mathbb R^{|\CP_p(\CTa)| \times |\BX_h^{\ba}|}$
after inverting the matrix representation of the left-hand side of
\eqref{eq_bra_lagrange_multiplier}.
Then, if $\TM$ and $\TK$
respectively denote the mass and stiffness matrices of
$\BX_h^{\ba} = \RT_p(\CTa) \cap \BH(\ddiv,\oa)$
and $\CP_p(\CTa)$, $\la$ can be obtained as the largest eigenvalue
of the problem: Find $\mu \in \mathbb R$ and ${\textup v} \in \mathbb R^{|\CP_p(\CTa)|}$
such that
\begin{equation*}
\TD^{\rm T} \TK \TD {\textup v} = \mu \TR^{\rm T} \TM \TR {\textup v}.
\end{equation*}

\subsection{A local mesh shape-regularity characterization}

Recall the notation of Section~\ref{sec_mesh}. For a vertex $\ba \in \CV_h$, define the constant
\begin{equation}\label{eq_ra}
\ra
\eq
1 + \frac{1}{\pi} \max_{K \in \CTa} \frac{h_K}{\haK}.
\end{equation}
It is fully computable and only depends on the shape-regularity parameter
$\kappa_{\CTa}$ of $\CTa$.

\subsection{Local-best approximation error estimate}

For the local-best approximation of~\eqref{eq_LB}, we observe that,
applying the Poincar\'e inequality stated in~\eqref{eq_poincare} repeatedly,
see~\cite[Corollary~12.13]{Ern_Guermond_FEs_I_21} or~\cite[Theorem~1.8.1]{Voh_FEM_LN_25}
for details,
\begin{equation}
\label{eq_repeated_poincare}
\|\bgrad(v-\bpi^p v)\|_{K}^2
\leq (s+1)!
\left (
\frac{h_K}{\pi}
\right )^{2s}
|v|_{H^{1+s}(K)}^2
\quad
\forall v \in H^{1+s}(K), \, \forall K \in \CT_h
\end{equation}
for all integer $0 \leq s \leq p$. Crucially, the constant appearing in~\eqref{eq_repeated_poincare} is fully explicit (computable).

\begin{remark}[Uniform-in-$s$ constant for a redefined Sobolev seminorm]
\label{rem_Hs}
If one redefines the seminorm $|v|_{H^s(\omega)}$ from~\eqref{eq_Hs_sn} to repetitively
contain the indices, starting with
$|v|_{H^2(\omega)} = \sum_{i=1}^d $ $\sum_{j=1}^d \|\partial_{\bx_j} \partial_{\bx_i} v\|_K^2$ in
place of $|v|_{H^2(\omega)} = \sum_{i=1}^d \sum_{j=i}^d \|\partial_{\bx_j} \partial_{\bx_i} v\|_K^2$,
then the factor $(s+1)!$ from~\eqref{eq_repeated_poincare} can be removed.
\end{remark}

\subsection{The quasi-interpolation operator $\CJ_h^p$ } \label{sec_int}

Let $u \in H^1_0(\Omega)$. To define our quasi-interpolation operator,
we employ~\eqref{eq_sha_min} with the argument $u_h$ being the local-best
approximation $\bpi^p u$ of the target function $u$ given by~\eqref{eq_LB},
restricted to the patch subdomain $\oa$:
\begin{equation}
\label{eq_definition_jj}
\CJ_h^p u \eq \sum_{\ba \in \CV_h} s_h^\ba((\bpi^p u)|_\oa).
\end{equation}
Since the computations of $\bpi^p u$ and $s_h^\ba((\bpi^p u)|_\oa)$
amount to element-wise/patch-wise finite element solves, $\CJ^p_h u$ is computable. This operator
is also local in the sense that for any element $K \in \CT_h$,
$\left . \left (\CJ_h^p u\right ) \right |_K$ only depends on the values
of $u$ in the vertex patch subdomains $\oa$ for the vertices $\ba \in \CV_K$.

\subsection{$H^1$ seminorm error estimate for $\CJ_h^p$ with computable constants}

We now state our main result on the approximation error of $\CJ_h^p u$ in the $H^1$ seminorm
(energy norm). We present it in a Pythagoras form that we are lead to by the orthogonal projection
property~\eqref{eq_LB}. Computable quasi-interpolation estimates follow from combination
with~\eqref{eq_repeated_poincare}. 

\begin{theorem}[$H^1$ seminorm error estimate with computable constants]
\label{thm_error_H1}
Let $u \in H^1_0(\Omega)$ and let the quasi-interpolation operator $\CJ_h^p$
be given by~\eqref{eq_LB}, \eqref{eq_sha}, \eqref{eq_definition_jj}. Then
\begin{align}
\label{eq_error_la_K}
\|\grad(u-\CJ_h^{p}(u))\|_K
&\leq
\left\{\|\bgrad(u-\bpi^p u)\|_K^{2}
+ \left(\sum_{\ba \in \CV_K} \ra \la \|\bgrad(u-\bpi^p u)\|_\oa\right)^{2}\right\}^{1/2}
\\
\nonumber
&\leq
\left(\frac{1}{(d+1)}+(d+1) c_K^{2}\right)^{1/2} \left\{\sum_{\ba \in \CV_K}\|\bgrad(u-\bpi^p u)\|_\oa^2\right\}^{1/2} \qquad \forall K \in \CT_h,
\end{align}
with the computable constant
\begin{equation*}
c_K \eq \max_{\ba \in \CV_K} ( \ra \la ).
\end{equation*}
There also holds
\begin{equation}
\label{eq_error_la_global}
\|\grad(u-\CJ_h^p u)\|_\Omega
\leq
(1+(d+1)^2c_\Omega^{2})^{1/2}\|\bgrad(u-\bpi^p u)\|_\Omega
\end{equation}
with the computable constant
\begin{equation*}
c_\Omega
\eq
\max_{\ba \in \CV_h} (\ra \la).
\end{equation*}
\end{theorem}

\subsection{$L^2$ error estimate for $\CJ_h^p$ with computable constants}

We state here our main result on the approximation error of $\CJ_h^p u$ in the $L^2$ norm.
We present it in a triangle-inequality form, since there is no orthogonal projection here.
Computable quasi-interpolation estimates follow from combination with~\eqref{eq_repeated_poincare}. 

\begin{theorem}[$L^2$ error estimate with computable constants]
\label{thm_error_L2}
Let $u \in H^1_0(\Omega)$ and let the quasi-interpolation operator $\CJ_h^p$
be given by~\eqref{eq_LB}, \eqref{eq_sha}, \eqref{eq_definition_jj}. Then
\begin{align}
\label{eq_error_L2_local}
\|u - \CJ_h^p u\|_K
&\leq
\frac{h_K}{\pi} \|\bgrad(u-\bpi^p u)\|_{K}
+
\frac{2}{d}
\sum_{\ba \in \CV_K} \ra\la h_{\ba} \| \bgrad(u-\bpi^p u)\|_{\oa}\\
\\
\nonumber
&
\leq
\left( \frac 1 \pi + \frac{2 c_K}{d}\right) \sum_{\ba \in \CV_K} h_{\ba} \| \bgrad(u-\bpi^p u)\|_{\oa} \qquad \forall K \in \CT_h.
\end{align}
There also holds
\begin{equation}
\label{eq_error_L2_global}
\|u-\CJ_h^p u\|_\Omega
\leq
\left (\frac{1}{\pi \sqrt{d+1}} + \frac{2}{d} \sqrt{d+1} c_\Omega \right )
\left\{\sum_{\ba \in \CV_h} h_{\ba}^2 \|\bgrad( u-\bpi^p u)\|_\oa^2\right\}^{1/2}.
\end{equation}
\end{theorem}

\subsection{Projection property}

From~\eqref{eq_error_la_global} or~\eqref{eq_error_L2_global},
when $u \in \CP_p(\CT_h)$, then $u=\bpi^p u$. Hence, we immediately have:

\begin{theorem}[Projection] \label{thm_proj}
The quasi-interpolation operator $\CJ_h^p$
from~\eqref{eq_LB}, \eqref{eq_sha}, \eqref{eq_definition_jj} is a projector, i.e.,
\[
    \CJ_h^p(u) = u \qquad \forall u \in \CP_p(\CT_h) \cap H^1_0(\Omega).
\]
\end{theorem}

\section{Proofs of the main results}
\label{section_proof}

\subsection{Preliminary results and notation}

For each vertex $\ba \in \CV_h$, there exists a finite set $\REF(\kappa_{\CTa})$
of reference patches only depending on the shape-regularity parameter
$\kappa_{\CTa}$ and space dimension $d$ and satisfying the following properties.
Each $\CTo \in \REF(\kappa_{\CTa})$
is a conforming mesh of simplices sharing the vertex $\bzero$, and $\po$ denotes the
piecewise affine ``hat function'' taking value $1$ at $\bzero$ and $0$ at all the other
vertices of $\CTo$; $\oo$ is the open domain corresponding to $\CTo$. All the elements
$\hK \in \CTo$ satisfy
\begin{equation*}
c(\kappa_{\CTa},d) \leq \rho_{\hK}, \qquad h_{\hK} \leq C(\kappa_{\CTa},d),
\end{equation*}
for two generic constants only depending on $\kappa_{\CTa}$ and $d$. In addition,
there exists a piecewise affine bilipchitz mapping $\phi: \oo \to \oa$
that transforms the elements of $\CTo$ into those of $\CTa$. We denote
by $\phig$ and $\phid$ the associated gradient- and divergence-preserving
Piola mappings,
$\phi^{\rm g}: L^2(\oo) \to L^2(\oa)$ and
$\phi^{\rm d}: \BL^2(\oo) \to \BL^2(\oa)$ respectively defined by
\begin{equation}
\label{eq_definition_piola}
\phi^{\rm g}(\widehat v)
\eq
\widehat v \circ \phi^{-1}
\quad \text{ and } \quad
\phi^{\rm d}(\widehat \bw)
\eq
\left (\frac{\JAC }{|\JAC|} \widehat \bw \right ) \circ \phi^{-1}
\end{equation}
for all $\widehat v \in L^2(\oo)$ and $\widehat \bw \in \BL^2(\oo)$, where $\JAC$
is the Jacobian matrix of $\phi$ and $|\JAC|$ is its determinant
(see, e.g., \cite[Section 7.2]{Ern_Guermond_FEs_I_21}). The following properties of Piola
mappings will be useful.
First, if $\hv \in H^1(\CTo)$ and $\widehat \bw \in \BL^2(\oo)$, we have
\begin{equation}
\label{eq_piola_stokes}
(\bgrad (\phig\hv),\phid\widehat \bw)_{\oa}
=
(\bgrad \hv,\widehat \bw)_{\oo}.
\end{equation}
The Piola mappings are invertible, and $\po \eq (\phig)^{-1} \pa$. We have
$\phig(\po\hv) = \pa\phig(\hv)$ and
\begin{equation}
\label{eq_bound_psig}
\|\bgrad (\phig \hv)\|_{\oa}
\leq
C(\kappa_{\CTa},d) \ha^{d/2-1}
\|\bgrad \hv\|_{\oo}
\quad
\end{equation}
for all $\hv \in H^1(\CTo)$, with $\ha$ introduced in~\eqref{eq_ha}.
Besides, $\phid$ is an isomorphism between $\RT_p(\CTo) \cap \BH(\ddiv^0,\oo)$ and
$\RT_p(\CTa) \cap \BH(\ddiv^0,\oa)$, and we have
\begin{equation}
\label{eq_bound_psid}
\|\phid \widehat \bw\|_{\oa}
\leq
C(\kappa_{\CTa},d) \ha^{1-d/2} \|\widehat \bw\|_{\oo}
\quad
\forall \widehat \bw \in \BL^2(\oo).
\end{equation}

We denote by $\CFo$ the faces of a reference patch $\CTo$. The subset
\begin{equation*}
\CFoi \eq \left \{
F \in \CFo \; | \; \po|_F \neq 0,
\right \}
\end{equation*}
coinciding with interior faces for an interior vertex, will be useful.
If $F \in \CFo$ is a face, $\CP_q(F)$ is the
set of functions $v: F \to \mathbb R$ that are polynomial of degree
less than or equal to $q$ in the ($d-1$) directions tangential to $F$.
For a collection of faces $\CF \subset \CFo$, we write
$\CP_q(\CF) \eq \cup_{F \in \CF} \CP_q(F)$.
We also associate with each face $F \in \CFo$
a unit normal vector $\bn_F$. If $F \subset \partial \oo$, we assume that
$\bn_F$ points outward $\oo$. For interior faces, the orientation of $\bn_F$
is arbitrary, but fixed. If $v \in \CP_{p+1}(\CTo)$, the jump of $v$ of through an
interior face $F = \partial \widehat K_- \cap \widehat \partial K_+ \in \CFo$ is defined by
\begin{equation*}
\jmp{v}_F \eq v_-|_F (\bn_- {\cdot} \bn_F) + v_+|_F (\bn_+ {\cdot} \bn_F)
\end{equation*}
where $v_\pm$ is the restriction of $v$ to $K_\pm$ and $\bn_\pm$
is the unit normal vector to $\partial K_\pm$ pointing outward $K_\pm$;
note that $\bn_\pm {\cdot} \bn_F = \pm 1$ only determines the sign.
For exterior faces $F \subset \partial \oo$, we simply set
\begin{equation*}
\jmp{v}_F = v|_F.
\end{equation*}
Thanks to these definitions, the integration by parts formula
\begin{equation}
\label{eq_stokes_DG}
(\bgrad v,\bw)_{\oo}
=
\sum_{F \in \CFo} (\jmp{v}_F,\bw {\cdot} \bn_F)_F
-
(v,\div \bw)_{\oo}
\end{equation}
holds true for all $\bw \in \RT_p(\CTo) \cap \BH(\ddiv,\oo)$ and $v \in \CP_{p+1}(\CTo)$.

\subsection{Proof of the properties of $\la$ from~\eqref{eq_lmbd}}

Here, we establish the following important result:

\begin{theorem}[Upper bound on $\la$]
\label{thm_la}
The constant $\la$ from~\eqref{eq_lmbd} is finite and only depends on the shape-regularity
parameter of $\CTa$, the space dimension $d$, and the polynomial degree $p$, i.e.,
\begin{equation}
\label{eq_bound_la}
\la \leq C(\kappa_{\CTa},d,p).
\end{equation}
In addition, we have
\begin{equation}
\label{eq_upper_bound_la}
\|\bgrad(\bII^p(\pa u_h)-s_h^\ba(u_h))\|_\oa
\leq
\la \min_{v \in H^1_0(\oa)} \|\bgrad(\pa u_h - v)\|_\oa
\qquad
\end{equation}
for all $u_h \in \CP_p(\CTa)$.
\end{theorem}

Before proving Theorem~\ref{thm_la}, we first establish a few intermediate results.
If $\CTo$ is a reference patch,
we use the notation $\widehat \BW^{\bzero}_h \eq \RT_p(\CTo) \cap \BH(\ddiv^0,\oo)$
for the counterpart to $\BW^{\ba}_h$ of~\eqref{eq_Wha}.
We first show that even though functions in $\widehat \BW_h^{\bzero}$
satisfy a divergence constraint, we are still free to assign their normal
traces on relevant faces.

\begin{lemma}[Traces of Raviart--Thomas elements]
\label{lemma_traces}
Let $\CTo$ be a reference patch.
For each $F \in \CFoi$, consider $q_F \in \CP_p(F)$.
Then, there exists $\bw_h \in \widehat \BW_h^{\bzero}$
such that $\bw_h {\cdot} \bn_F = q_F$ for all $F \in \CFoi$.
\end{lemma}

\begin{proof}
We will construct a suitable function $\bw_h \in \widehat \BW_h^{\bzero}$ element by element.
Let $K \in \CTo$ be fixed.  Note that given $r_{\CF} \in \CP_p(\CF_K)$ and 
$r_K \in \CP_p(K)$, there exists a function $\bw_{K} \in \RT_p(K)$ such that
$\bw_{K} {\cdot} \bn_K = r_{\CF}$ on $\partial K$ and $\div \bw_{K} = r_K$ in $K$ if and only if
the Neumann compatibility condition $(r_K,1)_K = (r_{\CF},1)_{\partial K}$ is satisfied. 
The key observation is that there is always a face $F^\star \in \CF_K$ such that
$F^\star \notin \CFoi$, i.e., $\po|_F = 0$: this is the face opposite to the vertex $\ba$.
Then, we select a boundary datum $r_{\CF} \in \CP_p(\CF_K)$ such that $r_{\CF}|_F = q_F$ for
all faces $F$ of $K$ such that $F \in \CFoi$ and, by fixing a suitable Neumann value on $F^\star$,
such that $(r_{\CF},1)_{\CF_K} = 0$. Finally, we are eligible to impose
$\bw_h|_K {\cdot} \bn_K = r_{\CF}$ on $\partial K$ and $\div (\bw_h|_K) = 0$.
\end{proof}

By carefully selecting their normal traces, we can then show that functions in
$\widehat \BW_h^{\bzero}$ (and, after mapping, in $\BW_h^{\ba}$)
can be successfully employed to measure the level of nonconformity of a piecewise polynomial
function.

\begin{corollary}[Control of the nonconformity]
\label{corollary_H1}
For all $u_h \in \CP_p(\CTo)$, if
\begin{equation*}
(\bgrad(\po u_h),\vv_h)_{\oo} = 0
\quad
\forall \vv_h \in \widehat \BW_h^{\bzero},
\end{equation*}
then $\po u_h \in H^1_0(\oo)$. In particular, $u_h \in H^1(\oo)$. 
\end{corollary}

\begin{proof}
Let $u_h \in \CP_p(\CTo)$. For all $\vv_h \in \widehat \BW_h^{\bzero}$,
since $\div \vv_h = 0$, Stokes' formula~\eqref{eq_stokes_DG} shows that
\begin{equation*}
(\bgrad(\po u_h),\vv_h)_{\oo}
=
\sum_{F \in \CFo}
(\jmp{\po u_h}_F,\vv_h {\cdot} \bn_F)_F
=
\sum_{F \in \CFoi}
(\po \jmp{u_h}_F,\vv_h {\cdot} \bn_F)_F,
\end{equation*}
since $\po$ is continuous and $\po|_F = 0$ for all
$F \in \CFo \setminus \CFoi$.
Following Lemma~\ref{lemma_traces}, we can now pick
$\vv_h^\star \in \widehat \BW_h^{\bzero}$
such that $\vv_h^\star {\cdot} \bn_F = \jmp{u_h}_F$ for all $F \in \CFoi$.
This gives
\begin{equation*}
\sum_{F \in \CFoi} \|\po^{1/2} \jmp{u_h}_F\|_F^2
=
(\bgrad(\po u_h),\vv_h^{\star})_{\oo}
=
0.
\end{equation*}
Since $\po > 0$ a.e. on each $F \in \CFoi$, this shows that $\jmp{u_h}_F = 0$
for all $F \in \CFoi$. Hence $u_h \in H^1(\oo)$. 
There also holds $u_h = 0$ on $F \subset \partial\oo$, $F \in \CFoi$
(such faces only exist when the vertex $\bzero$ lies on the boundary of $\oo$).
Thus, it immediately follows that $\po u_h \in H^1_0(\oo)$.
\end{proof}

We are now ready to establish our main result concerning $\la$.

\begin{proof}[Proof of Theorem~\ref{thm_la}]
Fix a vertex $\ba \in \CV_h$, and consider the associated reference patch
$\CTo \in \REF(\kappa_{\CTa})$.

If $\hu_h \in \CP_p(\CTo)$, we define $\hs_h^{\bzero}(\hu_h) \in \CP_p(\CTo) \cap H^1_0(\oo)$
by requiring that
\begin{equation} \label{eq_definition_s0}
(\grad \hs_h^{\bzero}(\hu_h),\grad \hv_h)_{\oo}
=
(\bgrad (\bII^p(\po \hu_h)),\grad \hv_h)_{\oo}
\qquad
\forall \hv_h \in \CP_p(\CTo) \cap H^1_0(\oo),
\end{equation}
and similarly we let $\widehat \br_h^{\bzero}(\hu_h)$ be the only element of
$\widehat \BW_h^{\bzero}$
such that
\begin{equation} \label{eq_br0}
(\widehat \br_h^{\bzero}(\hu_h),\widehat \bw_h)_{\oo}
=
(\bgrad(\po \hu_h),\widehat \bw_h)_{\oo} \qquad \forall \widehat \bw_h \in \widehat \BW_h^{\bzero}.
\end{equation}
These are the natural counterparts to $s_h^{\ba}(u_h)$ from~\eqref{eq_sha}
and $\br_h^{\ba}(u_h)$ from~\eqref{eq_bra} on the reference patch $\CTo$.

Now consider an arbitrary $u_h \in \CP_p(\CTa)$. Since $s_h^{\ba}(u_h)$
is the minimizer, see~\eqref{eq_sha_min}, we have
\begin{align*}
\|\bgrad(\bII^p(\pa u_h)-s_h^\ba(u_h))\|_{\oa}
&\leq
\|\bgrad(\bII^p(\pa u_h)-\phig(\hv_h))\|_{\oa}
\\
&\leq
C(\kappa_{\CTa},d) \ha^{d/2-1}
\|\bgrad(\widehat\CI_h^p(\po \hu_h)-\hv_h)\|_{\oo}
\end{align*}
for any $\hv_h \in \CP_p(\CTo) \cap H^1_0(\oo)$, with $\hu_h \eq (\phig)^{-1}(u_h) \in \CP_p(\CTo)$.
Here, we employed the fact that $\phig$ preserves piecewise polynomial
functions as well as point values in the Lagrange interpolation nodes
(affine-equivalent degrees of freedom) and used~\eqref{eq_bound_psig} in the last inequality.
We can minimize the last norm in the right-hand side, leading to
\begin{equation}
\label{tmp_bound_numerator}
\|\bgrad(\bII^p(\pa u_h)-s_h^\ba(u_h))\|_{\oa}
\leq
C(\kappa_{\CTa},d) \ha^{d/2-1}
\|\bgrad(\bII^p(\po \hu_h)-\hs_h^{\bzero}(\hu_h))\|_{\oo},
\end{equation}
by definition~\eqref{eq_definition_s0} of $\hs_h^{\bzero}(\hu_h)$.

On the other hand, still for the chosen $u_h \in \CP_p(\CTa)$, because of~\eqref{eq_piola_stokes},
for any $\widehat \bw_h \in \widehat \BW_h^{\bzero}$, we also have
\begin{equation} \label{eq_Piola}
(\bgrad(\pa u_h),\phid(\widehat \bw_h))_{\oa}
=
(\bgrad(\po\hu_h),\widehat \bw_h)_{\oo}.
\end{equation}
Moreover, as in~\eqref{eq_bra_sup}, from~\eqref{eq_br0}, there holds
\begin{equation}
\label{eq_bra_sup_0}
\|\widehat \br_h^{\bzero}(\hu_h)\|_{\oo}
=
\max_{\substack{\widehat \bw_h \in \widehat \BW_h^{\bzero} \\ \|\widehat \bw_h\|_{\oo} = 1}}
(\bgrad(\po \hu_h),\widehat \bw_h)_{\oo}.
\end{equation}
Thus, using the maximizer from~\eqref{eq_bra_sup_0},
$\widehat \bw_h^\star \in \widehat \BW_h^{\bzero}$ with $\|\widehat \bw_h^\star\|_{\oo} = 1$,
and employing $\phid(\widehat \bw_h^\star) /$ $\|\phid(\widehat \bw_h^\star)\|_{\oa}$
in~\eqref{eq_bra_sup}, 
we have
\begin{align*}
\|\br_h^\ba(u_h)\|_{\oa}
&\stackrel{\eqref{eq_bra_sup}}{\geq}
\frac{1}{\|\phid(\widehat \bw_h^\star)\|_{\oa}}
(\bgrad(\pa u_h),\phid(\widehat \bw_h^\star))_{\oa}
\\
& \stackrel{\eqref{eq_Piola}}{=}
\frac{1}{\|\phid(\widehat \bw_h^\star)\|_{\oa}}
(\bgrad(\po \hu_h),\widehat \bw_h^\star))_{\oo}
\stackrel{\eqref{eq_bra_sup_0}}{=}
\frac{1}{\|\phid(\widehat \bw_h^\star)\|_{\oa}}
\|\widehat \br_h^{\bzero}(\hu_h)\|_{\oo}.
\end{align*}
Finally, we use~\eqref{eq_bound_psid} to establish that
\begin{equation*}
\|\phid(\widehat \bw_h^\star)\|_{\oa}
\leq
C(\kappa_{\CTa},d) \ha^{1-d/2}\|\widehat \bw_h^\star\|_{\oo}
=
C(\kappa_{\CTa},d) \ha^{1-d/2},
\end{equation*}
which leads to
\begin{equation}
\label{tmp_bound_denumerator}
\frac{1}{\|\br_h^{\ba}(u_h)\|_{\oa}}
\leq
C(\kappa_{\CTa},d)\frac{\ha^{1-d/2}}{\|\widehat \br_h^{\bzero}(\hu_h)\|_{\oo}}.
\end{equation}
Combining~\eqref{tmp_bound_numerator} and~\eqref{tmp_bound_denumerator} gives
\begin{equation*}
\frac{\|\bgrad(\bII^p(\pa u_h)-s_h^{\ba}(u_h))\|_{\oa}}{\|\br_h^{\ba}(u_h)\|_{\oa}}
\leq
C(\kappa_{\CTa},d)
\frac{\|\bgrad(\bII^p(\po \hu_h)-\hs_h^{\bzero}(\hu_h))\|_{\oo}}{\|\widehat \br_h^{\bzero}(\hu_h)\|_{\oo}}
\leq
C(\kappa_{\CTa},d) \lo,
\end{equation*}
where
\begin{equation*}
\lo \eq \max_{\hu_h \in \CP_p(\CTo)}
\frac{\|\bgrad(\bII^p(\po \hu_h)-\hs_h^{\bzero}(\hu_h))\|_{\oo}}{\|\widehat \br_h^{\bzero}(\hu_h)\|_{\oo}}.
\end{equation*}
In other words, since the last inequality is valid for all $u_h \in \CP_p(\CTa)$,
we have actually shown that $\la \leq C(\kappa_{\CTa},d) \lo$.

We next show that $\lo$ is finite, only depending on the shape regularity $\kappa_{\CTa}$,
the space dimension $d$, and the polynomial degree $p$. Since $\CP_p(\CTo)$ is a finite-dimensional
vector space (with dimension only depending on $\kappa_{\CTa}$, $d$, and $p$),
it suffices to show that $\|\widehat \br_h^{\bzero}(\hu_h)\|_{\oo} = 0$
implies that $\|\bgrad(\bII^p(\po \hu_h)-\hs_h^{\bzero}(\hu_h))\|_{\oo} = 0$ for all
$\hu_h \in \CP_p(\CTo)$. Thus, we consider $\hu_h \in \CP_p(\CTo)$ such that
$\|\widehat \br_h^{\bzero}(\hu_h)\|_{\oo} = 0$. Due to~\eqref{eq_bra_sup_0} and Corollary
\ref{corollary_H1}, $\po \hu_h \in \CP_{p+1}(\CTo) \cap H^1_0(\oo)$. Then, using the
minimization property of $\hs_h^{\bzero}(\hu_h)$ that follows from~\eqref{eq_definition_s0},
cf.~\eqref{eq_sha_min}, we have
\begin{equation*}
\|\bgrad(\bII^p(\po \hu_h)-\hs_h^{\bzero}(u_h))\|_{\oo}
\leq
\|\bgrad(\bII^p(\po \hu_h)-\bII^p(\po \hu_h))\|_{\oo}
=
0,
\end{equation*}
since $\bII^p(\po \hu_h) \in \CP_p(\CTo) \cap H^1_0(\oo)$.
Thus, we have established~\eqref{eq_bound_la}.

It remains to establish the bound~\eqref{eq_upper_bound_la}. We first note that
\begin{equation*}
\|\bgrad(\bII^p(\pa u_h)-s_h^\ba(u_h))\|_\oa
\leq
\la \|\br^\ba(u_h)\|_\oa
\qquad
\forall u_h \in \CP_p(\CTa)
\end{equation*}
by the definition of $\la$. On the other hand, using~\eqref{eq_bra_sup}, we have
\begin{align*}
\|\br^\ba(u_h)\|_\oa
& =
\max_{\substack{\bw_h \in \BW_h^\ba \\ \|\bw_h\|_\oa = 1}}
(\bgrad(\pa u_h),\bw_h)_\oa
=
\max_{\substack{\bw_h \in \BW_h^\ba \\ \|\bw_h\|_\oa = 1}}
(\bgrad(\pa u_h-v),\bw_h)_\oa \\
& \leq
\|\bgrad(\pa u_h-v)\|_\oa
\end{align*}
for all $v \in H^1_0(\oa)$, since we always have
\begin{equation*}
(\bgrad v,\bw_h)_\oa
=
(\grad v,\bw_h)_\oa
=
- (v,\div \bw_h)_\oa = 0
\end{equation*}
for all $\bw_h \in \BW_h^\ba$.
\end{proof}

\subsection{Proof of $H^1$ seminorm error estimates of Theorem~\ref{thm_error_H1}}

In this section, we now derive the local and global error estimates
in the $H^1$ norm stated in Theorem~\ref{thm_error_H1}. We start with a
fundamental result concerning the local contributions $s_h^{\ba}$.

\begin{lemma}[Patch-wise estimates]\label{lem_globloc_patch_la}
The estimate
\begin{equation}
\label{eq_globloc_patch_la}
\|\bgrad(\bII^p(\pa \bpi^p u)-s_h^\ba((\bpi^p u)|_\oa))\|_\oa
\leq
\ra \la \|\bgrad(u-\bpi^p u)\|_\oa
\end{equation}
holds true for each vertex $\ba \in \CV_h$.
\end{lemma}

\begin{proof}
We start with~\eqref{eq_upper_bound_la}. Next, since $\pa u \in H^1_0(\oa)$, we have
\begin{equation*}
\min_{v \in H^1_0(\oa)} \|\bgrad(\pa \bpi^p u-v)\|_\oa
\leq
\|\bgrad(\pa (\bpi^p u-u))\|_\oa.
\end{equation*}
Recalling that $(\bpi^p u,1)_K = (u,1)_K$ for all $K \in \CTa$ from~\eqref{eq_LB},
we then employ the product rule element-wise to show that
\begin{align*}
\|\grad(\pa (u-\bpi^p u))\|_K
&\leq
\|\grad \pa\|_{\BL^\infty(K)}\|u-\bpi^p u\|_K + \|\pa\|_{L^\infty(K)}\|\grad(u-\bpi^p u)\|_K
\\
&\leq
\left (1 + \frac{h_K}{\pi\haK}\right )
\|\grad(u-\bpi^p u)\|_K,
\end{align*}
where we employed~\eqref{eq_pa} and~\eqref{eq_poincare}. The desired
result~\eqref{eq_globloc_patch_la} follows by the definition of $\ra$ in~\eqref{eq_ra}.
\end{proof}

\begin{proof}[Proof of Theorem~\ref{thm_error_H1}]
From~\eqref{eq_LB_K}, the Pythagoras equality yields
\begin{equation*}
\|\grad(u-\CJ_h^p(u))\|_K^{2}
=
\|\grad(u-\bpi^p u)\|_K^{2} + \|\grad(\bpi^p u-\CJ_h^{p}(u))\|_K^{2}.
\end{equation*}
The partition of unity~\eqref{eq_PU} and the projection and linearity properties of $\bII^p$ give
\begin{equation}
\label{eq_Ip_PU}
\sum_{\ba \in \CV_h} \bII^p(\pa\bpi^p u)
=
\bII^p\Bigg(\sum_{\ba \in \CV_h} (\pa\bpi^p u)\Bigg)
=
\bII^p(\bpi^p u)
=
\bpi^p u.
\end{equation}
Thus, together with definition~\eqref{eq_definition_jj} and Lemma~\ref{lem_globloc_patch_la}, we estimate the second term as
\begin{align*}
\|\grad(\bpi^p u - \CJ_h^p(u))\|_K
&=
\left \|
\sum_{\ba \in \CV_K} \grad(\bII^p(\pa\bpi^p u) - s_h^\ba((\bpi^p u)|_\oa))
\right \|_K
\\
&\leq
\sum_{\ba \in \CV_K}\|\grad(\bII^p(\pa\bpi^p u) - s_h^\ba((\bpi^p u)|_\oa))\|_K
\\
&\leq
\sum_{\ba \in \CV_K}\|\bgrad(\bII^p(\pa\bpi^p u) - s_h^\ba((\bpi^p u)|_\oa))\|_\oa
\\
&\leq
\sum_{\ba \in \CV_K}\ra\la\|\bgrad(u-\bpi^p u)\|_\oa,
\end{align*}
and the first inequality in~\eqref{eq_error_la_K} immediately follows.

We also have
\begin{align*}
\left (\sum_{\ba \in \CV_K}\ra\la\|\bgrad(u-\bpi^p u)\|_\oa \right )^2
&\leq
(d+1)\sum_{\ba \in \CV_K}(\ra\la)^2\|\bgrad(u-\bpi^p u)\|_\oa^2
\\
&\leq
(d+1)\max_{\ba \in \CV_K} (\ra\la)^2\sum_{\ba \in \CV_K}\|\bgrad(u-\bpi^p u)\|_\oa^2,
\end{align*}
which proves the second inequality in~\eqref{eq_error_la_K}.

Finally, to establish~\eqref{eq_error_la_global}, we treat sharply
\begin{equation}
\label{eq_Om_gr}
\begin{split}
\|\bgrad(\bpi^p u - \CJ_h^p(u))\|_\Omega^2
&=
\sum_{K \in \CT_h} \|\grad(\bpi^p u - \CJ_h^p(u))\|_K^2
\\
&\leq
\sum_{K \in \CT_h} (d+1) \sum_{\ba \in \CV_K}\|\grad(\bII^p(\pa\bpi^p u) - s_h^\ba((\bpi^p u)|_\oa))\|_K^2
\\
&=
(d+1) \sum_{\ba \in \CV_h}\|\bgrad(\bII^p(\pa\bpi^p u) - s_h^\ba((\bpi^p u)|_\oa))\|_\oa^2
\\
&\leq
(d+1) \sum_{\ba \in \CV_h} (\ra\la)^2 \|\bgrad(u-\bpi^p u)\|_\oa^2
\\
&\leq
\max_{\ba \in \CV_h} (\ra\la)^2(d+1)^2\|\bgrad(u-\bpi^p u)\|_\Omega^2.
\end{split}\end{equation}
\end{proof}

\subsection{Proof of $L^2$ error estimates of Theorem~\ref{thm_error_L2}}

We finally establish the local and global error estimates presented
in Theorem~\ref{thm_error_L2}. We start by establishing
a Poincar\'e inequality.

\begin{lemma}[Element-wise Poincar\'e]
\label{lemma_poincare_elem}
Consider an element $K \in \CT_h$, one of its vertices $\ba \in \CV_K$
and the face $F \in \CF_K$ opposite $\ba$. Then, if $v \in H^1(K)$
satisfies $v = 0$ on $F$, we have
\begin{equation}
\label{eq_poincare_elem}
\|v\|_K
\leq
\frac{2h_K}{d} \|\grad v\|_K.
\end{equation}
\end{lemma}

\begin{proof}
Let $\by: \bx \to \bx-\ba$. Since $\div \by = d$, we have
\begin{equation*}
d \int_K |v|^2
=
\int_K \div \by |v|^2
=
\int_{\partial K} \by {\cdot} \bn|v|^2
-
2 \int_K \by {\cdot} \grad v v.
\end{equation*}
We then observe that the boundary term vanishes as $\by {\cdot} \bn = 0$
on all the faces $F' \in \CF_K$ sharing the vertex $\ba$, and $v = 0$
on the remaining face $F$. We then use that $|\by| \leq h_K$,
and therefore
\begin{equation*}
d\|v\|_K^2 \leq 2h_K \|\grad v\|_K\|v\|_K
\end{equation*}
using the Cauchy--Schwarz inequality. The result follows.
\end{proof}

\begin{corollary}[Closeness to the local-best approximation in $L^2$]
\label{cor_error_L2}
For all $\ba \in \CV_h$, we have
\begin{equation}
\label{eq_error_L2_patch}
\|\bII^p(\pa\bpi^p u)-s_h^{\ba}(\bpi^p u)|_\oa\|_{\oa}
\leq
\frac{2}{d} h_{\ba} 
\|\bgrad(\bII^p(\pa\bpi^p u)-s_h^{\ba}(\bpi^p u)|_\oa)\|_{\oa}.
\end{equation}
In addition, the estimates
\begin{equation}
\label{eq_error_L2_elem}
\|\bpi^p u - \CJ_h^p u\|_K
\leq
\frac{2}{d}
\sum_{\ba \in \CV_K} \ra\la h_{\ba} \| \bgrad(u-\bpi^p u)\|_{\oa}
\end{equation}
and
\begin{equation}
\label{eq_error_L2_domain}
\|\bpi^p u - \CJ_h^p u\|_\Omega
\leq
\frac{2}{d} \sqrt{d+1} c_\Omega
\left\{\sum_{\ba \in \CV_h} h_{\ba}^2 \|\bgrad( u-\bpi^p u)\|_\oa^2\right\}^{1/2}
\end{equation}
hold true.
\end{corollary}

\begin{proof}
The estimate in~\eqref{eq_error_L2_patch} simply follows by applying
Lemma~\ref{lemma_poincare_elem} in each element of the patch,
noting that, for each $K \in \CTa$, both $\bII^p(\pa\bpi^p u)$ and $s_h^{\ba}(\bpi^p u)|_\oa$
vanish on the face of $K$ opposite to $\ba$, see~\eqref{eq_Ip} and~\eqref{eq_sha}. 

For the second estimate~\eqref{eq_error_L2_elem}, we observe
from~\eqref{eq_Ip_PU} and~\eqref{eq_definition_jj} that
\begin{equation*}
(\bpi^p u - \CJ_h^p u)|_K
=
\sum_{\ba \in \CV_K} \left \{\bII^p(\pa \bpi^p u) - s_h^{\ba}(\bpi^p u)|_\oa\right \}|_K
\end{equation*}
and therefore, using the triangle inequality together with~\eqref{eq_ha}
and~\eqref{eq_globloc_patch_la},
\begin{align*}
\|\bpi^p u - \CJ_h^p u\|_K
&\leq
\sum_{\ba \in \CV_K}\|\bII^p(\pa \bpi^p u) - s_h^{\ba}(\bpi^p u)|_\oa\|_K\\
&\leq
\frac{2}{d}
\sum_{\ba \in \CV_K} h_K \| \bgrad(\bII^p(\pa \bpi^p u) - s_h^{\ba}(\bpi^p u)|_\oa)\|_{K},\\
&\leq
\frac{2}{d}
\sum_{\ba \in \CV_K} h_{\ba} \|\bgrad(\bII^p(\pa \bpi^p u) - s_h^{\ba}(\bpi^p u)|_\oa)\|_{\oa} \\
&\leq \frac{2}{d} \sum_{\ba \in \CV_K} h_{\ba} \ra \la \|\bgrad(u-\bpi^p u)\|_\oa.
\end{align*}

To establish~\eqref{eq_error_L2_domain}, we estimate sharply, as in~\eqref{eq_Om_gr},
\begin{align*}
\|\bpi^p u - \CJ_h^p u\|_\Omega^2
&\stackrel{\mathmakebox[\widthof{x.xx}]{}}{\leq}
(d+1) \sum_{\ba \in \CV_h}\|\bII^p(\pa\bpi^p u) - s_h^\ba((\bpi^p u)|_\oa)\|_\oa^2
\\
&\stackrel{\eqref{eq_error_L2_patch}}{\leq}
\frac{4}{d^2} (d+1) \sum_{\ba \in \CV_h} h_{\ba}^2
\| \bgrad( \bII^p(\pa\bpi^p u) - s_h^\ba((\bpi^p u)|_\oa))\|_\oa^2
\\
&\stackrel{\eqref{eq_globloc_patch_la}}{\leq}
\frac{4}{d^2} (d+1) \max_{\ba \in \CV_h}(\ra\la)^2
\sum_{\ba \in \CV_h} h_{\ba}^2 \|\bgrad( u-\bpi^p u)\|_\oa^2.
\end{align*}
\end{proof}

\begin{proof}[Proof of Theorem~\ref{thm_error_L2}]
The estimates follow from the ones in Corollary~\ref{cor_error_L2},
triangle inequalities, and the fact that
\begin{equation*}
\|u-\bpi^p u\|_{K} \leq \frac{h_K}{\pi} \|\grad(u-\bpi^p u)\|_{K}
\end{equation*}
for all $K \in \CT_h$.
For~\eqref{eq_error_L2_global}, we namely use $\|u-\CJ_h^p u\|_\Omega \leq \|u-\bpi^p u  \|_\Omega+ \|\bpi^p u - \CJ_h^p u\|_\Omega$ together with~\eqref{eq_error_L2_domain} and 
\begin{align*}
\|u-\bpi^p u\|_\Omega^2 & \leq \sum_{K \in \CT_h} \frac{h_K^2}{\pi^2} \|\grad(u-\bpi^p u)\|_K^2 = 
\frac{1}{d+1}\sum_{\ba \in \CV_h} \sum_{K \in \CT_h | \ba \in \CV_K} \frac{h_K^2}{\pi^2} \|\grad(u-\bpi^p u)\|_K^2 \\
& \leq \frac{1}{\pi^2(d+1)} \sum_{\ba \in \CV_h} h_{\ba}^2 \|\bgrad( u-\bpi^p u)\|_\oa^2.
\end{align*}
\end{proof}

\section{Numerical examples} \label{sec_num}

We present here the results of several numerical experiments
illustrating the actual practical behavior of our quasi-interpolate.

\subsection{Setting}

For different two-dimensional domains $\Omega \subset \mathbb R^2$,
functions $u \in H^1_0(\Omega) \cap C^0(\overline{\Omega})$ with various additional regularity,
and triangular meshes $\CT_h$, we compute the (broken) local-best approximation
$\bpi^1 u \in \CP_1(\CT_h)$, 
\begin{equation*}
\bpi^1 u \eq \arg \min_{v_h \in \CP_1(\CT_h)} \|\bgrad(u-v_h)\|_\Omega,
\end{equation*}
the global-best approximation
\begin{equation*}
u_h^1 \eq \arg \min_{v_h \in \CP_1(\CT_h) \cap H^1_0(\Omega)} \|\grad(u-u_h^1)\|_\Omega,
\end{equation*}
the Lagrange interpolant $\bII^1 u$ of Section~\ref{sec_Ip}, and our projection $\CJ_h^1u$
of Section~\ref{sec_int}, both for the polynomial degree $p=1$.
Except from the local-best approximation, all these approximations of $u$
are conforming and sit in the finite element space $\CP_1(\CT_h) \cap H^1_0(\Omega)$.

We employ both structured and unstructured meshes $\CT_h$. Our structured
meshes are simply constructed by first building a Cartesian grid of squares, and
then breaking each square into four triangles by connected its vertices to its barycenter.
The unstructured meshes are generated by the {\tt mmg} software package \cite{dobrzynski_2012a}.
We use both uniform and adapted unstructured meshes. In the uniform case, the meshes are
simply generated using the {\tt -hmax} flag of {\tt mmg}, whereas the local mesh size is
specified through a ``metic file'' via the {\tt -sol} flag for adapted meshes.

For a given mesh $\CT_h$, $N \eq \dim \CP_1(\CT_h) \cap H^1_0(\Omega)$ denotes the number of
degrees of freedom of the associated finite element space. For a uniform mesh with
maximal size $h$, we have $N \sim (\operatorname{diam}(\Omega)/h)^2$.

For each configuration of domain and function, our results are summarized in a figure
with 4 panels. Panels (a) and (b) aim at comparing the quality of approximation of different
projectors. Specifically, the projection errors measured in the $H^1$ seminorm
and in the $L^2$ norm are respectively represented on panels (a) and (b).
On the other hand, in panels (c) and (d), we measure the quality of
our guaranteed upper bounds, using that the local-best approximation error
$\|\bgrad(u-\bpi^1 u)\|_\Omega$ is known. On panel (c), we plot
$\|\grad(u-\CJ_h^1 u)\|_\Omega$ and $\|u-\CJ_h^1 u\|_\Omega$ together with the
right-hand sides in the guaranteed estimations in~\eqref{eq_error_la_global}
and~\eqref{eq_error_L2_global}, written as
\begin{equation*}
\|\grad(u-\CJ_h^1 u)\|_\Omega
\leq
(1+(d+1)^2c_\Omega^{2})^{1/2}\|\bgrad(u-\bpi^1 u)\|_\Omega =: \eta_{H^1}(\bpi^1 u)
\end{equation*}
and
\begin{equation*}
\|u-\CJ_h^1 u\|_\Omega
\leq
\left (\frac{1}{\pi \sqrt{d+1}} + \frac{2}{d} \sqrt{d+1} c_\Omega \right )
\left\{\sum_{\ba \in \CV_h} h_{\ba}^2 \|\bgrad( u-\bpi^1 u)\|_\oa^2\right\}^{1/2} =: \eta_{L^2}(\bpi^1 u).
\end{equation*}
Finally, we plot the ratios on panel (d), thereby measuring the overestimation
factor in our guaranteed bounds.

We obtain similar results in all the test cases below, so that we give a discussion
in Section~\ref{section_numerics_comments}.

\subsection{Smooth function}

We start with the square $\Omega \eq (-1,1)^2$ and the smooth function
\begin{equation*}
u(\bx) = \sin(\pi\bx_1)\sin(\pi\bx_2).
\end{equation*}
We employ uniform unstructured meshes generated by {\tt mmg}.
The result are presented in Figure~\ref{figure_smooth}.

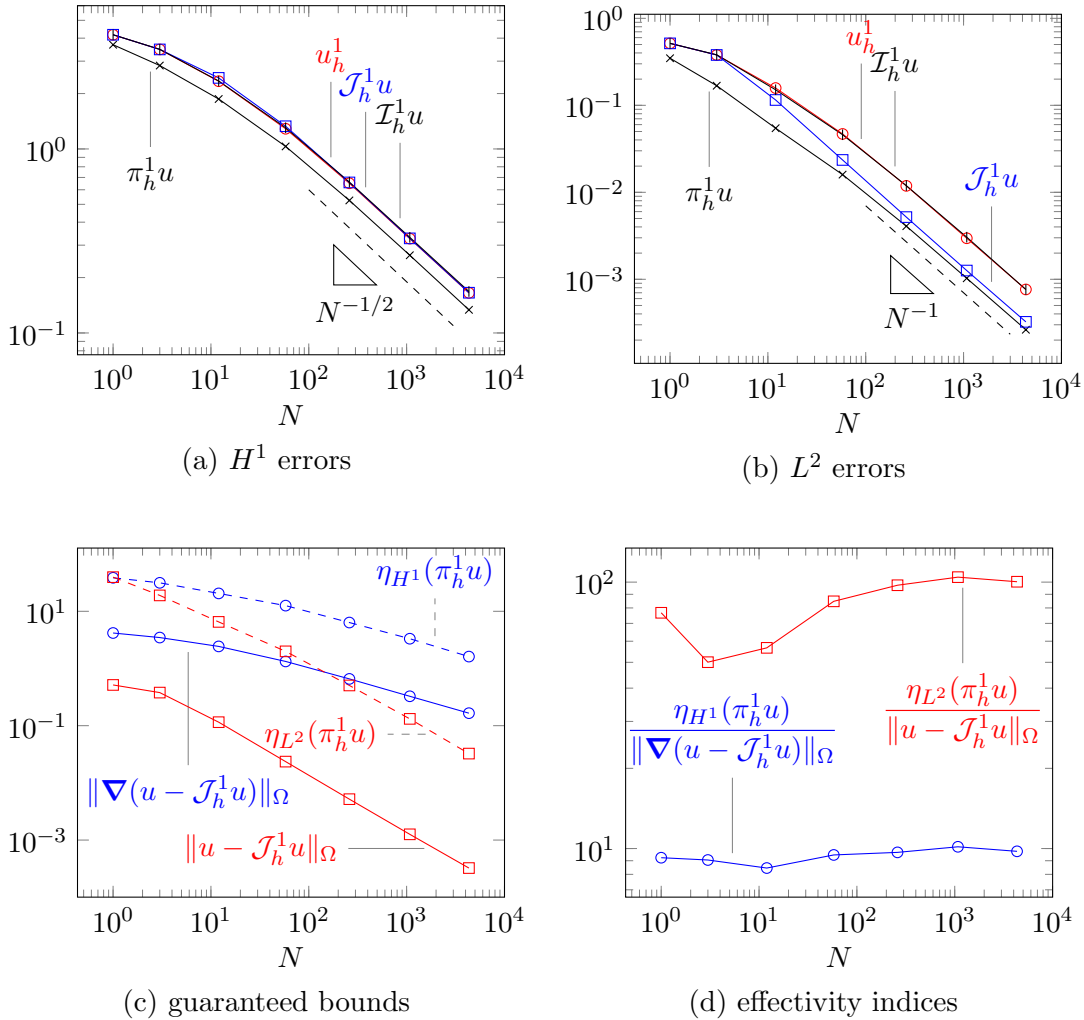
\begin{figure}
\begin{minipage}{.45\linewidth}
\centering
\begin{tikzpicture}
\begin{axis}
[
	width=\linewidth,
	xmode=log,
	ymode=log,
	xlabel={$N$}
]

\plot[black,mark=x     ] table[x=nrdofs,y=LB] {figures/smooth/data/errors_H1.txt}
	node[pos=0.1,pin={[pin distance=1.0cm]-90:{$\bpi^1 u$}}] {};
\plot[red  ,mark=o     ] table[x=nrdofs,y=GB] {figures/smooth/data/errors_H1.txt}
	node[pos=0.6,pin={[pin distance=1.0cm]90:{$u_h^1$}}] {};
\plot[blue ,mark=square] table[x=nrdofs,y=QI] {figures/smooth/data/errors_H1.txt}
	node[pos=0.7,pin={[pin distance=1.0cm]90:{$\CJ_h^1 u$}}] {};
\plot[black,mark=|     ] table[x=nrdofs,y=LI] {figures/smooth/data/errors_H1.txt}
	node[pos=0.8,pin={[pin distance=1.0cm]90:{$\bII^1 u$}}] {};

\SlopeTriangle{0.6}{-0.1}{0.2}{-0.5}{$N^{-1/2}$}{};
\plot[black,dashed,domain=1.e2:3.e3] {6*x^(-0.5)};

\end{axis}
\end{tikzpicture}

(a) $H^1$ errors
\end{minipage}
\begin{minipage}{.45\linewidth}
\centering
\begin{tikzpicture}
\begin{axis}
[
	width=\linewidth,
	xmode=log,
	ymode=log,
	xlabel={$N$}
]

\plot[black,mark=x     ] table[x=nrdofs,y=LB ] {figures/smooth/data/errors_L2.txt}
	node[pos=0.1,pin={[pin distance=1.0cm]-90:{$\bpi^1 u$}}] {};
\plot[red  ,mark=o     ] table[x=nrdofs,y=GB ] {figures/smooth/data/errors_L2.txt}
	node[pos=0.5,pin={[pin distance=1.0cm]90:{$u_h^1$}}] {};
\plot[blue ,mark=square] table[x=nrdofs,y=QI ] {figures/smooth/data/errors_L2.txt}
	node[pos=0.9,pin={[pin distance=1.0cm]90:{$\CJ_h^1 u$}}] {};
\plot[black,mark=|     ] table[x=nrdofs,y=LI ] {figures/smooth/data/errors_L2.txt}
	node[pos=0.6,pin={[pin distance=1.0cm]90:{$\bII^1 u$}}] {};

\SlopeTriangle{0.6}{-0.1}{0.2}{-1}{$N^{-1}$}{};
\plot[black,dashed,domain=1.e2:3.e3] {0.7*x^(-1)};

\end{axis}
\end{tikzpicture}

(b) $L^2$ errors
\end{minipage}

\vspace*{0.8cm}

\begin{minipage}{.45\linewidth}
\centering
\begin{tikzpicture}
\begin{axis}
[
	width=\linewidth,
	xmode=log,
	ymode=log,
	xlabel={$N$}
]

\plot[blue,mark=o     ,mark options={solid}       ] table[x=nrdofs,y=eH1] {figures/smooth/data/ratios.txt}
	node[pos=0.2,pin={[pin distance=1.5cm]-90:{$\|\grad(u-\CJ_h^1 u)\|_\Omega$}}] {};
\plot[blue,mark=o     ,mark options={solid},dashed] table[x=nrdofs,y=EH1] {figures/smooth/data/ratios.txt}
	node[pos=0.9,pin={[pin distance=0.5cm]90:{$\eta_{H^1}(\bpi^1 u)$}}] {};
\plot[red ,mark=square,mark options={solid}       ] table[x=nrdofs,y=eL2] {figures/smooth/data/ratios.txt}
	node[pos=0.9,pin={[pin distance=1cm]180:{$\|u-\CJ_h^1 u\|_\Omega$}}] {};
\plot[red ,mark=square,mark options={solid},dashed] table[x=nrdofs,y=EL2] {figures/smooth/data/ratios.txt}
	node[pos=0.9,pin={[pin distance=0.5cm]180:{$\eta_{L^2}(\bpi^1 u)$}}] {};

\end{axis}
\end{tikzpicture}

(c) guaranteed bounds 
\end{minipage}
\begin{minipage}{.45\linewidth}
\centering
\begin{tikzpicture}
\begin{axis}
[
	width=\linewidth,
	xmode=log,
	ymode=log,
	xlabel={$N$}
]

\plot[blue,mark=o     ] table[x=nrdofs,y expr=\thisrow{EH1}/\thisrow{eH1}] {figures/smooth/data/ratios.txt}
	node[pos=0.2,pin={[pin distance=1.0cm]90:{$\displaystyle \frac{\eta_{H^1}(\bpi^1 u)}{\|\grad(u-\CJ_h^1 u)\|_\Omega}$}}] {};
\plot[red ,mark=square] table[x=nrdofs,y expr=\thisrow{EL2}/\thisrow{eL2}] {figures/smooth/data/ratios.txt}
	node[pos=0.85,pin={[pin distance=1.0cm]-90:{$\displaystyle \frac{\eta_{L^2}(\bpi^1 u)}{\|u-\CJ_h^1 u\|_\Omega}$}}] {};

\end{axis}
\end{tikzpicture}

(d) effectivity indices
\end{minipage}
\caption{Smooth solution example}
\label{figure_smooth}
\end{figure}

\subsection{Circular irregularity}

We keep the domain $\Omega \eq (-1,1)^2$ and consider the function
\begin{equation*}
u(\bx) = (1-|\bx|) \mathbf 1_{|\bx| < 1}.
\end{equation*}
This function is $C^0$, but not $C^1$. In fact, the gradient has a
line of discontinuity across the unit circle, so that $h^{1/2}$
and $h^{3/2}$ (i.e. $N^{-1/4}$ and $N^{-3/4}$) rates are expected
for the $H^1$ and $L^2$ errors.  We employ uniform structured meshes.
Figure~\ref{figure_circle} summarizes the results.

\begin{figure}
\begin{minipage}{.45\linewidth}
\centering
\begin{tikzpicture}
\begin{axis}
[
	width=\linewidth,
	xmode=log,
	ymode=log,
	xlabel={$N$}
]

\plot[black,mark=x     ] table[x=nrdofs,y=LB] {figures/circle/data/errors_H1.txt}
	node[pos=0.1,pin={[pin distance=1.0cm]-90:{$\bpi^1 u$}}] {};
\plot[red  ,mark=o     ] table[x=nrdofs,y=GB] {figures/circle/data/errors_H1.txt}
	node[pos=0.6,pin={[pin distance=1.0cm]90:{$u_h^1$}}] {};
\plot[blue ,mark=square] table[x=nrdofs,y=QI] {figures/circle/data/errors_H1.txt}
	node[pos=0.7,pin={[pin distance=1.0cm]90:{$\CJ_h^1 u$}}] {};
\plot[black,mark=|     ] table[x=nrdofs,y=LI] {figures/circle/data/errors_H1.txt}
	node[pos=0.8,pin={[pin distance=1.0cm]90:{$\bII^1 u$}}] {};

\SlopeTriangle{0.5}{-0.15}{0.2}{-0.25}{$N^{-1/4}$}{};
\plot[black,dashed,domain=1.e2:3.e3] {1*x^(-0.25)};

\end{axis}
\end{tikzpicture}

(a) $H^1$ errors
\end{minipage}
\begin{minipage}{.45\linewidth}
\centering
\begin{tikzpicture}
\begin{axis}
[
	width=\linewidth,
	xmode=log,
	ymode=log,
	xlabel={$N$}
]

\plot[black,mark=x     ] table[x=nrdofs,y=LB ] {figures/circle/data/errors_L2.txt}
	node[pos=0.1,pin={[pin distance=1.0cm]-90:{$\bpi^1 u$}}] {};
\plot[red  ,mark=o     ] table[x=nrdofs,y=GB ] {figures/circle/data/errors_L2.txt}
	node[pos=0.5,pin={[pin distance=1.0cm]90:{$u_h^1$}}] {};
\plot[blue ,mark=square] table[x=nrdofs,y=QI ] {figures/circle/data/errors_L2.txt}
	node[pos=0.9,pin={[pin distance=1.0cm]90:{$\CJ_h^1 u$}}] {};
\plot[black,mark=|     ] table[x=nrdofs,y=LI ] {figures/circle/data/errors_L2.txt}
	node[pos=0.6,pin={[pin distance=1.0cm]90:{$\bII^1 u$}}] {};

\SlopeTriangle{0.55}{-0.15}{0.2}{-0.75}{$N^{-3/4}$}{};
\plot[black,dashed,domain=1.e2:3.e3] {0.2*x^(-0.75)};

\end{axis}
\end{tikzpicture}

(b) $L^2$ errors
\end{minipage}

\vspace*{0.8cm}

\begin{minipage}{.45\linewidth}
\centering
\begin{tikzpicture}
\begin{axis}
[
	width=\linewidth,
	xmode=log,
	ymode=log,
	xlabel={$N$}
]

\plot[blue,mark=o     ,mark options={solid}       ] table[x=nrdofs,y=eH1] {figures/circle/data/ratios.txt}
	node[pos=0.2,pin={[pin distance=1.5cm]-90:{$\|\grad(u-\CJ_h^1 u)\|_\Omega$}}] {};
\plot[blue,mark=o     ,mark options={solid},dashed] table[x=nrdofs,y=EH1] {figures/circle/data/ratios.txt}
	node[pos=0.9,pin={[pin distance=0.5cm]90:{$\eta_{H^1}(\bpi^1 u)$}}] {};
\plot[red ,mark=square,mark options={solid}       ] table[x=nrdofs,y=eL2] {figures/circle/data/ratios.txt}
	node[pos=0.9,pin={[pin distance=1cm]180:{$\|u-\CJ_h^1 u\|_\Omega$}}] {};
\plot[red ,mark=square,mark options={solid},dashed] table[x=nrdofs,y=EL2] {figures/circle/data/ratios.txt}
	node[pos=0.9,pin={[pin distance=0.5cm]180:{$\eta_{L^2}(\bpi^1 u)$}}] {};

\end{axis}
\end{tikzpicture}

(c) guaranteed bounds 
\end{minipage}
\begin{minipage}{.45\linewidth}
\centering
\begin{tikzpicture}
\begin{axis}
[
	width=\linewidth,
	xmode=log,
	ymode=log,
	xlabel={$N$}
]

\plot[blue,mark=o     ] table[x=nrdofs,y expr=\thisrow{EH1}/\thisrow{eH1}] {figures/circle/data/ratios.txt}
	node[pos=0.2,pin={[pin distance=1.0cm]90:{$\displaystyle \frac{\eta_{H^1}(\bpi^1 u)}{\|\grad(u-\CJ_h^1 u)\|_\Omega}$}}] {};
\plot[red ,mark=square] table[x=nrdofs,y expr=\thisrow{EL2}/\thisrow{eL2}] {figures/circle/data/ratios.txt}
	node[pos=0.85,pin={[pin distance=1.0cm]-90:{$\displaystyle \frac{\eta_{L^2}(\bpi^1 u)}{\|u-\CJ_h^1 u\|_\Omega}$}}] {};

\end{axis}
\end{tikzpicture}

(d) effectivity indices
\end{minipage}
\caption{Circular interface solution example}
\label{figure_circle}
\end{figure}
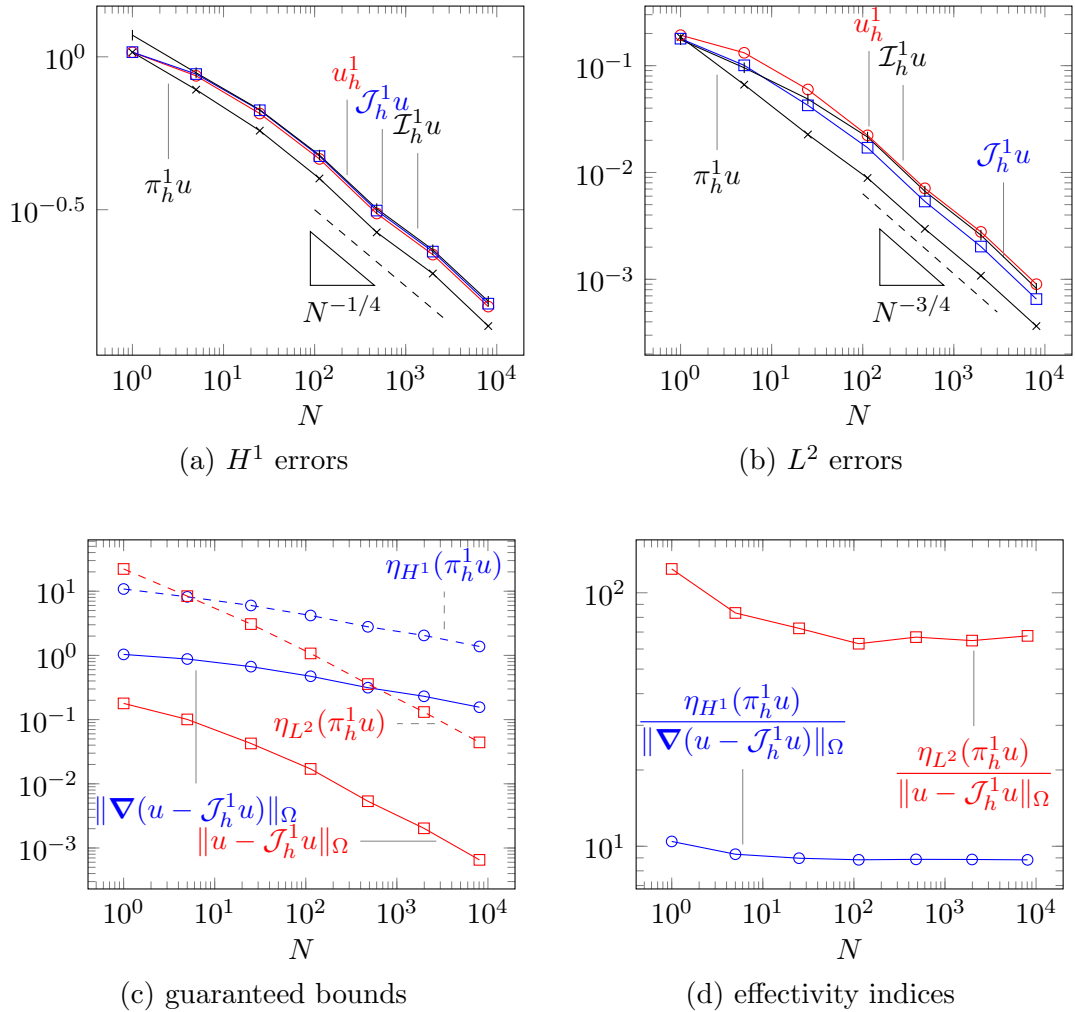

\subsection{Corner singularity}

We now consider the $L$-shape domain $\Omega \eq (-1,1)^2 \setminus [0,1]^2$,
and the standard corner singularity function
\begin{equation*}
u(\bx) = \chi(\bx)|\bx|^\alpha \sin(\alpha\theta(\bx))
\end{equation*}
where $\alpha \eq 2/3$ and $\theta(\bx)$ is an angular variable
defined such that $\theta = 0$ for $\bx \in (0,1) \times \{0\}$
and $3\pi/2$ when $\bx \in \{0\} \times (0,1)$ and
\begin{equation*}
\chi(\bx) \eq (1-|\bx|^2) \mathbf 1_{|\bx| < 1}
\end{equation*}
is a cutoff function that enforces boundary conditions. We first use
uniform meshes generated by {\tt mmg} and present the corresponding
results in Figure~\ref{figure_lshape}. In this case, we expect and observe convergence
rates of order $N^{-1/3}$ and $N^{-5/6}$ in the $H^1$ and $L^2$ norm.
We then employed adapted meshes (still generated by {\tt mmg}) with corresponding
results presented in Figure~\ref{figure_lshape_adapted}. Optimal convergence
rates are expected and observed in this case. The adapted meshes are generated by requiring
that $h_K \lesssim \max(h_{\rm max}^{3/2},|\bx_K|^{1/3} h_{\rm max})$ for all
$K \in \CT_h$, where $\bx_K$ is the barycenter of $K$.

\begin{figure}
\begin{minipage}{.45\linewidth}
\centering
\begin{tikzpicture}
\begin{axis}
[
	width=\linewidth,
	xmode=log,
	ymode=log,
	xlabel={$N$}
]

\plot[black,mark=x     ] table[x=nrdofs,y=LB] {figures/lshape/data/errors_H1.txt}
	node[pos=0.1,pin={[pin distance=1.0cm]-90:{$\bpi^1 u$}}] {};
\plot[red  ,mark=o     ] table[x=nrdofs,y=GB] {figures/lshape/data/errors_H1.txt}
	node[pos=0.6,pin={[pin distance=1.0cm]90:{$u_h^1$}}] {};
\plot[blue ,mark=square] table[x=nrdofs,y=QI] {figures/lshape/data/errors_H1.txt}
	node[pos=0.7,pin={[pin distance=1.0cm]90:{$\CJ_h^1 u$}}] {};
\plot[black,mark=|     ] table[x=nrdofs,y=LI] {figures/lshape/data/errors_H1.txt}
	node[pos=0.8,pin={[pin distance=1.0cm]90:{$\bII^1 u$}}] {};

\SlopeTriangle{0.5}{-0.15}{0.2}{-0.333333}{$N^{-1/3}$}{};
\plot[black,dashed,domain=1.e2:3.e3] {0.4*x^(-0.33333333)};

\end{axis}
\end{tikzpicture}

(a) $H^1$ errors
\end{minipage}
\begin{minipage}{.45\linewidth}
\centering
\begin{tikzpicture}
\begin{axis}
[
	width=\linewidth,
	xmode=log,
	ymode=log,
	xlabel={$N$}
]

\plot[black,mark=x     ] table[x=nrdofs,y=LB ] {figures/lshape/data/errors_L2.txt}
	node[pos=0.1,pin={[pin distance=1.0cm]-90:{$\bpi^1 u$}}] {};
\plot[red  ,mark=o     ] table[x=nrdofs,y=GB ] {figures/lshape/data/errors_L2.txt}
	node[pos=0.9,pin={[pin distance=0.5cm]-90:{$u_h^1$}}] {};
\plot[blue ,mark=square] table[x=nrdofs,y=QI ] {figures/lshape/data/errors_L2.txt}
	node[pos=0.9,pin={[pin distance=0.5cm]90:{$\CJ_h^1 u$}}] {};
\plot[black,mark=|     ] table[x=nrdofs,y=LI ] {figures/lshape/data/errors_L2.txt}
	node[pos=0.6,pin={[pin distance=0.5cm]90:{$\bII^1 u$}}] {};

\SlopeTriangle{0.55}{-0.15}{0.15}{-0.83333333333333333333}{$N^{-5/6}$}{};
\plot[black,dashed,domain=1.e2:3.e3] {0.08*x^(-.83333333333333333333)};

\end{axis}
\end{tikzpicture}

(b) $L^2$ errors
\end{minipage}

\vspace*{0.8cm}

\begin{minipage}{.45\linewidth}
\centering
\begin{tikzpicture}
\begin{axis}
[
	width=\linewidth,
	xmode=log,
	ymode=log,
	xlabel={$N$}
]

\plot[blue,mark=o     ,mark options={solid}       ] table[x=nrdofs,y=eH1] {figures/lshape/data/ratios.txt}
	node[pos=0.2,pin={[pin distance=1.5cm]-90:{$\|\grad(u-\CJ_h^1 u)\|_\Omega$}}] {};
\plot[blue,mark=o     ,mark options={solid},dashed] table[x=nrdofs,y=EH1] {figures/lshape/data/ratios.txt}
	node[pos=0.9,pin={[pin distance=0.5cm]90:{$\eta_{H^1}(\bpi^1 u)$}}] {};
\plot[red ,mark=square,mark options={solid}       ] table[x=nrdofs,y=eL2] {figures/lshape/data/ratios.txt}
	node[pos=0.9,pin={[pin distance=1cm]180:{$\|u-\CJ_h^1 u\|_\Omega$}}] {};
\plot[red ,mark=square,mark options={solid},dashed] table[x=nrdofs,y=EL2] {figures/lshape/data/ratios.txt}
	node[pos=0.9,pin={[pin distance=0.5cm]180:{$\eta_{L^2}(\bpi^1 u)$}}] {};

\end{axis}
\end{tikzpicture}

(c) guaranteed bounds 
\end{minipage}
\begin{minipage}{.45\linewidth}
\centering
\begin{tikzpicture}
\begin{axis}
[
	width=\linewidth,
	xmode=log,
	ymode=log,
	xlabel={$N$}
]

\plot[blue,mark=o     ] table[x=nrdofs,y expr=\thisrow{EH1}/\thisrow{eH1}] {figures/lshape/data/ratios.txt}
	node[pos=0.2,pin={[pin distance=1.0cm]90:{$\displaystyle \frac{\eta_{H^1}(\bpi^1 u)}{\|\grad(u-\CJ_h^1 u)\|_\Omega}$}}] {};
\plot[red ,mark=square] table[x=nrdofs,y expr=\thisrow{EL2}/\thisrow{eL2}] {figures/lshape/data/ratios.txt}
	node[pos=0.85,pin={[pin distance=1.0cm]-90:{$\displaystyle \frac{\eta_{L^2}(\bpi^1 u)}{\|u-\CJ_h^1 u\|_\Omega}$}}] {};

\end{axis}
\end{tikzpicture}

(d) effectivity indices
\end{minipage}
\caption{Singular function example}
\label{figure_lshape}
\end{figure}
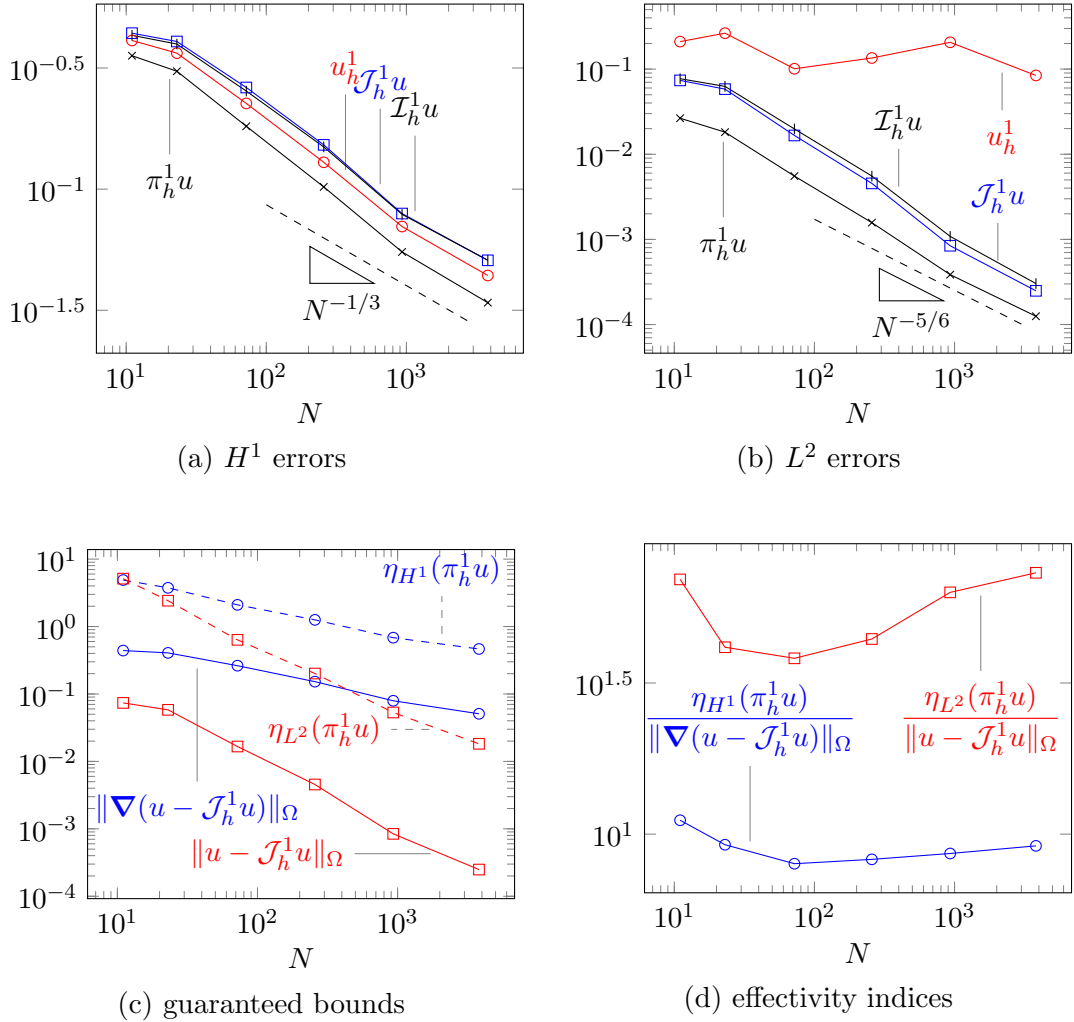

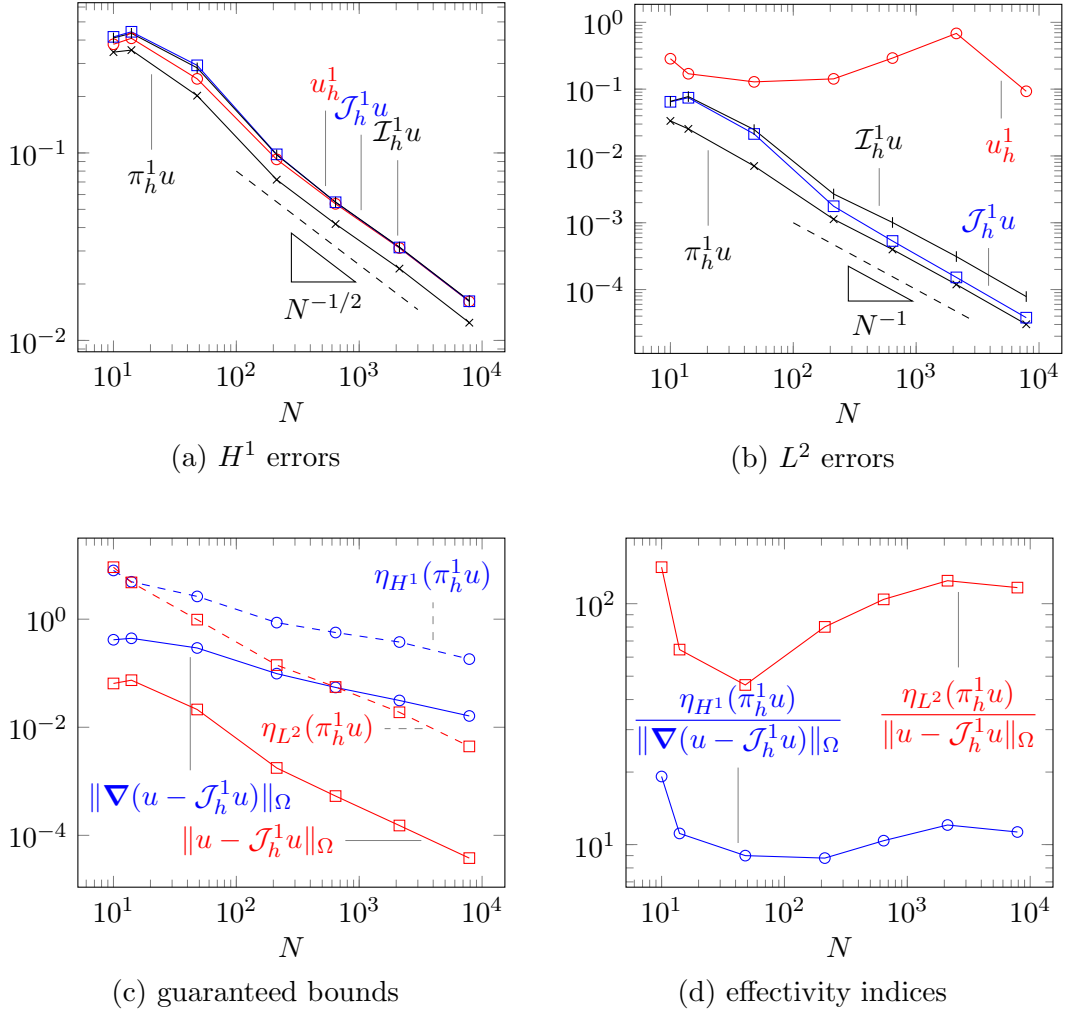
\begin{figure}
\begin{minipage}{.45\linewidth}
\centering
\begin{tikzpicture}
\begin{axis}
[
	width=\linewidth,
	xmode=log,
	ymode=log,
	xlabel={$N$}
]

\plot[black,mark=x     ] table[x=nrdofs,y=LB] {figures/lshape_adapted/data/errors_H1.txt}
	node[pos=0.1,pin={[pin distance=1.0cm]-90:{$\bpi^1 u$}}] {};
\plot[red  ,mark=o     ] table[x=nrdofs,y=GB] {figures/lshape_adapted/data/errors_H1.txt}
	node[pos=0.6,pin={[pin distance=1.0cm]90:{$u_h^1$}}] {};
\plot[blue ,mark=square] table[x=nrdofs,y=QI] {figures/lshape_adapted/data/errors_H1.txt}
	node[pos=0.7,pin={[pin distance=1.0cm]90:{$\CJ_h^1 u$}}] {};
\plot[black,mark=|     ] table[x=nrdofs,y=LI] {figures/lshape_adapted/data/errors_H1.txt}
	node[pos=0.8,pin={[pin distance=1.0cm]90:{$\bII^1 u$}}] {};

\SlopeTriangle{0.5}{-0.15}{0.2}{-0.5}{$N^{-1/2}$}{};
\plot[black,dashed,domain=1.e2:3.e3] {0.8*x^(-0.5)};

\end{axis}
\end{tikzpicture}

(a) $H^1$ errors
\end{minipage}
\begin{minipage}{.45\linewidth}
\centering
\begin{tikzpicture}
\begin{axis}
[
	width=\linewidth,
	xmode=log,
	ymode=log,
	xlabel={$N$}
]

\plot[black,mark=x     ] table[x=nrdofs,y=LB ] {figures/lshape_adapted/data/errors_L2.txt}
	node[pos=0.1,pin={[pin distance=1.0cm]-90:{$\bpi^1 u$}}] {};
\plot[red  ,mark=o     ] table[x=nrdofs,y=GB ] {figures/lshape_adapted/data/errors_L2.txt}
	node[pos=0.9,pin={[pin distance=0.5cm]-90:{$u_h^1$}}] {};
\plot[blue ,mark=square] table[x=nrdofs,y=QI ] {figures/lshape_adapted/data/errors_L2.txt}
	node[pos=0.9,pin={[pin distance=0.5cm]90:{$\CJ_h^1 u$}}] {};
\plot[black,mark=|     ] table[x=nrdofs,y=LI ] {figures/lshape_adapted/data/errors_L2.txt}
	node[pos=0.6,pin={[pin distance=0.5cm]90:{$\bII^1 u$}}] {};

\SlopeTriangle{0.5}{-0.15}{0.15}{-1}{$N^{-1}$}{};
\plot[black,dashed,domain=1.e2:3.e3] {0.1*x^(-1)};

\end{axis}
\end{tikzpicture}

(b) $L^2$ errors
\end{minipage}

\vspace*{0.8cm}

\begin{minipage}{.45\linewidth}
\centering
\begin{tikzpicture}
\begin{axis}
[
	width=\linewidth,
	xmode=log,
	ymode=log,
	xlabel={$N$}
]

\plot[blue,mark=o     ,mark options={solid}       ] table[x=nrdofs,y=eH1] {figures/lshape_adapted/data/ratios.txt}
	node[pos=0.2,pin={[pin distance=1.5cm]-90:{$\|\grad(u-\CJ_h^1 u)\|_\Omega$}}] {};
\plot[blue,mark=o     ,mark options={solid},dashed] table[x=nrdofs,y=EH1] {figures/lshape_adapted/data/ratios.txt}
	node[pos=0.9,pin={[pin distance=0.5cm]90:{$\eta_{H^1}(\bpi^1 u)$}}] {};
\plot[red ,mark=square,mark options={solid}       ] table[x=nrdofs,y=eL2] {figures/lshape_adapted/data/ratios.txt}
	node[pos=0.9,pin={[pin distance=1cm]180:{$\|u-\CJ_h^1 u\|_\Omega$}}] {};
\plot[red ,mark=square,mark options={solid},dashed] table[x=nrdofs,y=EL2] {figures/lshape_adapted/data/ratios.txt}
	node[pos=0.9,pin={[pin distance=0.5cm]180:{$\eta_{L^2}(\bpi^1 u)$}}] {};

\end{axis}
\end{tikzpicture}

(c) guaranteed bounds 
\end{minipage}
\begin{minipage}{.45\linewidth}
\centering
\begin{tikzpicture}
\begin{axis}
[
	width=\linewidth,
	xmode=log,
	ymode=log,
	xlabel={$N$}
]

\plot[blue,mark=o     ] table[x=nrdofs,y expr=\thisrow{EH1}/\thisrow{eH1}] {figures/lshape_adapted/data/ratios.txt}
	node[pos=0.25,pin={[pin distance=1.0cm]90:{$\displaystyle \frac{\eta_{H^1}(\bpi^1 u)}{\|\grad(u-\CJ_h^1 u)\|_\Omega}$}}] {};
\plot[red ,mark=square] table[x=nrdofs,y expr=\thisrow{EL2}/\thisrow{eL2}] {figures/lshape_adapted/data/ratios.txt}
	node[pos=0.85,pin={[pin distance=1.0cm]-90:{$\displaystyle \frac{\eta_{L^2}(\bpi^1 u)}{\|u-\CJ_h^1 u\|_\Omega}$}}] {};

\end{axis}
\end{tikzpicture}

(d) effectivity indices
\end{minipage}
\caption{Singular function example with adapted meshes}
\label{figure_lshape_adapted}
\end{figure}

\subsection{Discussion}
\label{section_numerics_comments}

In all the benchmarks, we obtain the expected rates of convergence
for our quasi-interpolation operator $\CJ_h^1$. 
In the $H^1$ seminorm, the errors of the global-best approximation $u_h^1$,
the Lagrange interpolant $\bII^1 u$, and our quasi-interpolant $\CJ_h^1 u$ are always similar. 
In the $L^2$ norm, our quasi-interpolant $\CJ_h^1 u$ is always more accurate than the Lagrange
interpolant $\bII^1 u$ and the global-best approximation $u_h^1$, and sometimes significantly so.
In both the $H^1$ seminorm and the $L^2$ norm, the local-best approximation
provides better errors than the conforming approximation, but the improvement
is typically marginal. Our guaranteed upper bounds never underestimate the errors
as predicted by the theory of Theorems~\ref{thm_error_H1} and~\ref{thm_error_L2},
and behave with the correct rates. They allow for error certification. The overestimation
factor is remarkably stable in all tested situations, with values around $10$ in the
$H^1$ seminorm and $100$ in the $L^2$ norm.

\bibliographystyle{amsplain}
\bibliography{bibliography.bib}

\end{document}